\documentclass[12pt]{article}
\usepackage[]{amsmath,amssymb}
\usepackage{amscd}
\usepackage{latexsym}
\usepackage{cite}
\usepackage{amsthm}

\newtheorem{definition}{Definition}[section]
\newtheorem{theorem}[definition]{Theorem}
\newtheorem{lemma}[definition]{Lemma}
\newtheorem{corollary}[definition]{Corollary}
\newtheorem{remark}[definition]{Remark}
\newtheorem{example}[definition]{Example}

\newtheorem{problem}[definition]{Problem}
\newtheorem{note}[definition]{Note}

\newtheorem{proposition}[definition]{Proposition}

\begin{document}
\title{\bf 
The alternating PBW basis for 
the \\
positive part of
$U_q(\widehat{\mathfrak{sl}}_2)$ }
\author{
Paul Terwilliger 
}
\date{}

\maketitle
\begin{abstract}
The positive part $U^+_q$ of 
$U_q(\widehat{\mathfrak{sl}}_2)$ has a presentation with
two generators $A,B$ that satisfy the cubic $q$-Serre relations.
We introduce a PBW basis for $U^+_q$, said to be alternating. 
Each element of this PBW basis commutes with exactly one of
$A$, $B$, $qAB-q^{-1}BA$. This gives three types of
PBW basis elements;
the  elements of each type
 mutually commute.
We interpret the alternating PBW basis in terms of
a $q$-shuffle algebra associated with affine $\mathfrak{sl}_2$.
We show how the alternating PBW basis is related
to the PBW basis for $U^+_q$ found by Damiani in 1993.

\bigskip


 \end{abstract}
\section{Introduction}
This paper is motivated by a recent development in statistical
mechanics, concerning the $q$-Onsager algebra $\mathcal O_q$
\cite{bas1},
 \cite{qSerre}.
 In \cite{BK05} Baseilhac and Koizumi introduced a current algebra
$\mathcal A_q$ for $\mathcal O_q$, in order to solve boundary
integrable systems with hidden symmetries. In
\cite[Definition~3.1]{basnc} Baseilhac and Shigechi give
a presentation of $\mathcal A_q$ by generators and relations.
The generators are denoted
$\lbrace \mathcal W_{-k}\rbrace_{k=0}^\infty$,
$\lbrace \mathcal W_{k+1}\rbrace_{k=0}^\infty$,
$\lbrace \mathcal G_{k+1}\rbrace_{k=0}^\infty$,
$\lbrace \mathcal {\tilde G}_{k+1}\rbrace_{k=0}^\infty$.
The relations involve $q$ and a nonzero 
scalar parameter $\rho$. In an attempt
to understand $\mathcal A_q$ we considered the
limiting case $\rho=0$. For this value of $\rho$
the algebra $\mathcal O_q$ gets replaced by an algebra $U^+_q$ called 
the positive part of the quantum group $U_q(\widehat{\mathfrak{sl}}_2)$.
The algebra $U^+_q$ has a presentation with two generators $A$, $B$
that satisfy the cubic $q$-Serre relations; 
see Definition \ref{def:posp} below. In this paper we display
some elements in $U^+_q$, denoted
\begin{align}
\label{eq:WWGG}
\lbrace  W_{-k}\rbrace_{k=0}^\infty, \quad 
\lbrace  W_{k+1}\rbrace_{k=0}^\infty, \quad
\lbrace  G_{k+1}\rbrace_{k=0}^\infty, \quad
\lbrace {\tilde G}_{k+1}\rbrace_{k=0}^\infty,
\end{align}
that satisfy
the $\rho=0$ analog of the relations in
\cite[Definition~3.1]{basnc}. These relations are 
given in 
Propositions \ref{prop:rel1},
\ref{prop:rel2} below; see
Propositions
\ref{prop:rel3},
\ref{prop:attract},
\ref{prop:GGWW}
for some additional relations.
\medskip

\noindent 
We defined the elements
(\ref{eq:WWGG}) and obtained the above relations
in the following way.
Start with the free algebra $\mathbb V$ on two generators $x,y$.
The standard (linear) basis for
$\mathbb V$ consists of the words in $x,y$.
In 
\cite{rosso1, rosso} M. Rosso introduced
an algebra structure on $\mathbb V$, called a
$q$-shuffle algebra.
For $u,v\in \lbrace x,y\rbrace$ their
$q$-shuffle product is
$u\star v = uv+q^{\langle u,v\rangle }vu$, where
$\langle u,v\rangle =2$
(resp. $\langle u,v\rangle =-2$)
if $u=v$ (resp.
 $u\not=v$).
Rosso gave an injective algebra homomorphism $\natural$ 
from $U^+_q$ into the $q$-shuffle algebra
${\mathbb V}$, that sends $A\mapsto x$ and $B\mapsto y$.
Let $U$ denote the image of $U^+_q$ under $\natural$.
A word $v_1v_2\cdots v_n$ in $\mathbb V$ is said to be
alternating whenever $n\geq 1$ and $v_{i-1}\not=v_i$ for
$2 \leq i \leq n$. We name the alternating words as follows:
\begin{align*}
&W_0 = x, \qquad W_{-1} = xyx, \qquad W_{-2} = xyxyx, \qquad \ldots
\\
&W_1 = y, \qquad W_{2} = yxy, \qquad W_{3} = yxyxy, \qquad \ldots 
\\
&G_{1} = yx, \qquad G_{2} = yxyx,  \qquad G_3 = yxyxyx, \qquad \ldots 
\\
&\tilde G_{1} = xy, \qquad \tilde G_{2} =
xyxy,\qquad \tilde G_3 = xyxyxy, \qquad \ldots
\end{align*}
We describe the $q$-shuffle product of every pair of alternating words.
Using this description we 
show that the alternating words satisfy the relations mentioned
below  (\ref{eq:WWGG}).
Using these relations we show that $U$ contains
the alternating words.
\medskip

\noindent We use the alternating words to obtain some
PBW bases for $U$.
For instance, we show that the elements
$\lbrace  W_{-k}\rbrace_{k=0}^\infty$,
$\lbrace  W_{k+1}\rbrace_{k=0}^\infty$,
$\lbrace {\tilde G}_{k+1}\rbrace_{k=0}^\infty$ 
(in appropriate linear order) give a PBW basis for $U$,
said to be alternating.
The elements 
$\lbrace  W_{-k}\rbrace_{k=0}^\infty$,
$\lbrace  W_{k+1}\rbrace_{k=0}^\infty$,
$\lbrace G_{k+1}\rbrace_{k=0}^\infty$ give a similar PBW basis
for $U$.
We describe how the alternating PBW basis is related to the
PBW basis for $U^+_q$ found by Damiani 
in \cite{damiani}. 
\medskip

\noindent This paper is organized as follows.
In Section 2, we recall the notion of a PBW basis, and
describe the one for $U^+_q$ found by Damiani.
In Section 3 we obtain some slightly technical facts about
$U^+_q$ that will be used later in the paper.
In Section 4 we describe the algebra homomorphism  $\natural $
from $U^+_q$ into
the $q$-shuffle algebra $\mathbb V$.
In Section 5 we introduce the alternating words in $\mathbb V$,
and obtain some relations involving these words.
In Sections 6, 7 we use these relations to obtain a commutator
relation for every pair of alternating words.
In Section 8 we obtain some additional relations for the alternating
words, which get used to show that the alternating words
are contained in $U$.
In Section 9 the alternating words are related using
generating functions.
In Section 10 we use the alternating words to 
obtain some PBW bases for $U$, including
the alternating PBW basis.
In Section 11 we show how the
alternating PBW basis is related to the Damiani PBW basis.
In Section 12 we give some slightly technical 
comments about some relations in Sections 6, 8.
In Section 13 we give some open problems.
In Appendices A, B we present the commutator relations 
in an alternative way.
In Appendix C we give some examples of the commutator relations.

\section{The algebra $U^+_q$}

\noindent We now begin our formal argument.
Recall the natural numbers $\mathbb N=\lbrace 0,1,2,\ldots\rbrace$
and integers $\mathbb Z = \lbrace 0,\pm 1, \pm 2,\ldots \rbrace$.
Let $\mathbb F$ denote a field.
We will be discussing vector spaces, tensor products,
and algebras.
Each vector space and tensor product discussed is over $\mathbb F$.
Each algebra discussed is associative, over $\mathbb F$,
and has a multiplicative identity.
A subalgebra has the same multiplicative identity as the parent algebra.

\begin{definition} 
\label{def:PBW}
\rm (See  \cite[p.~299]{damiani}.) 
Let $\mathcal A$ denote an algebra.
     A {\it Poincar\'e-Birkhoff-Witt} (or {\it PBW})
   basis for $\mathcal A$ consists of a subset $\Omega \subseteq
    \mathcal A$ and
   a linear order $<$ on $\Omega$, such that  the
   following is a basis for the vector space $\mathcal A$:
    \begin{align*}
    a_1 a_2 \cdots a_n \qquad \quad  n\in \mathbb N, \quad \qquad
    a_1,a_2,\ldots, a_n \in \Omega,
    \qquad \quad  a_1 \leq a_2 \leq \cdots \leq a_n.
    \end{align*}
  We interpret the empty product as the multiplicative identity in 
   $\mathcal A$.
   \end{definition}

\noindent Fix a nonzero $q \in \mathbb F$
that is not a root of unity.
Recall the notation
\begin{eqnarray*}
\lbrack n\rbrack_q = \frac{q^n-q^{-n}}{q-q^{-1}}
\qquad \qquad n \in \mathbb Z.
\end{eqnarray*}
\noindent For elements $X, Y$ in any algebra, define their
commutator and $q$-commutator by 
\begin{align*}
\lbrack X, Y \rbrack = XY-YX, \qquad \qquad
\lbrack X, Y \rbrack_q = q XY- q^{-1}YX.
\end{align*}
\noindent Note that 
\begin{align*}
\lbrack X, \lbrack X, \lbrack X, Y\rbrack_q \rbrack_{q^{-1}} \rbrack
= 
X^3Y-\lbrack 3\rbrack_q X^2YX+ 
\lbrack 3\rbrack_q XYX^2 -YX^3.
\end{align*}

\begin{definition}
\label{def:posp}
\rm (See \cite[Corollary~3.2.6]{lusztig}.) Define the algebra $U^+_q$ 
by generators $A,B$ and relations
\begin{eqnarray}
&&
\lbrack A, \lbrack A, \lbrack A, B\rbrack_q \rbrack_{q^{-1}} \rbrack=0,
\qquad \qquad 
\lbrack B, \lbrack B, \lbrack B, A\rbrack_q \rbrack_{q^{-1}}
\rbrack=0.
\label{eq:qSerre1}
\end{eqnarray}
\noindent We call $U^+_q$ the {\it positive part of 
$U_q(\widehat{\mathfrak{sl}}_2)$}.
The relations (\ref{eq:qSerre1})
are called the {\it $q$-Serre relations}.
\end{definition}

\noindent
For an algebra $\mathcal A$,
by an {\it automorphism} of $\mathcal A$ we mean an
algebra isomorphism $\mathcal A \to \mathcal A$.
By an {\it antiautomorphism} of $\mathcal A$, we mean an
$\mathbb F$-linear bijection $\gamma: \mathcal A \to \mathcal A$
such that $(ab)^\gamma= b^\gamma a^\gamma$ for all $a,b\in \mathcal A$.
\begin{lemma}
\label{lem:AAut}
There exists a unique automorphism $\sigma$ of
$U^+_q$ that swaps $A, B$. There exists a unique antiautomorphism
$S$ of $U^+_q$ that fixes each of $A$, $B$.
\end{lemma}

\noindent 
In 
\cite[p.~299]{damiani}  Damiani obtained a 
PBW basis for $U^+_q$, involving some elements
  \begin{equation}
   \label{eqUq:PBWintro}
    \lbrace E_{n \delta + \alpha_0}\rbrace_{n=0}^\infty,
     \qquad \quad
      \lbrace E_{n \delta + \alpha_1}\rbrace_{n=0}^\infty,
       \qquad \quad
        \lbrace E_{n \delta}\rbrace_{n=1}^\infty.
         \end{equation}
These elements are recursively defined as follows:
\begin{align}
E_{\alpha_0} = A, \qquad \qquad
E_{\alpha_1} = B, \qquad \qquad
E_{\delta} = q^{-2}BA-AB,
\label{eq:BAalt}
\end{align}
and for $n\geq 1$,
\begin{align}
&
E_{n \delta+\alpha_0} =
\frac{
\lbrack E_\delta, E_{(n-1)\delta+ \alpha_0} \rbrack
}
{q+q^{-1}},
\qquad \qquad
E_{n \delta+\alpha_1} =
 \frac{
 \lbrack
 E_{(n-1)\delta+ \alpha_1},
 E_\delta
 \rbrack
 }
 {q+q^{-1}},
 \label{eq:dam1introalt}
 \\
 &
 \qquad \qquad
 E_{n \delta} =
 q^{-2}  E_{(n-1)\delta+\alpha_1} A
 - A E_{(n-1)\delta+\alpha_1}.
 \label{eq:dam2introalt}
\end{align}


\begin{proposition}
    \label{prop:PBWbasis}
    {\rm (See \cite[p.~308]{damiani}.)}
The elements 
   {\rm (\ref{eqUq:PBWintro})} in the 
     linear
    order
 \begin{align*}
 E_{\alpha_0} < E_{\delta+\alpha_0} <
  E_{2\delta+\alpha_0}
  < \cdots
  <
   E_{\delta} < E_{2\delta}
    < E_{3\delta}
 < \cdots
  <
   E_{2\delta + \alpha_1} <
     E_{\delta + \alpha_1} < E_{\alpha_1}
     \end{align*}
    form a PBW basis for $U^+_q$.
   \end{proposition}

\noindent 
The PBW basis elements 
   (\ref{eqUq:PBWintro}) are known to satisfy certain relations
  \cite[Section~4]{damiani}. For instance  
the elements $\lbrace E_{n\delta}\rbrace_{n=1}^\infty$ mutually
commute
  \cite[p.~307]{damiani}.
\medskip

\noindent Next we describe a grading for the algebra $U^+_q$.
Note that the $q$-Serre relations are homogeneous in both
$A$ and $B$. Therefore
the algebra $U^+_q$ has a ${\mathbb N}^2$-grading 
for which $A$ and $B$ are homogeneous,
with degrees $(1,0)$ and $(0,1)$ respectively.
For this grading the 
PBW basis elements 
   (\ref{eqUq:PBWintro}) are homogeneous with degrees
   shown below:
\bigskip

\centerline{
\begin{tabular}[t]{c|c}
{\rm PBW basis element} & {\rm degree} 
   \\  \hline
   $ E_{n \delta + \alpha_0}$ & $(n+1,n)$
	 \\
   $ E_{n \delta + \alpha_1}$ & $(n,n+1)$
          \\
   $ E_{n \delta}$ & $(n,n)$
	 \end{tabular}
              }
              \medskip

\noindent 
Using this data and Proposition
    \ref{prop:PBWbasis}, one can obtain the dimension
    of each homogeneous component for 
the $\mathbb N^2$-grading of $U^+_q$.
This calculation 
is elementary but useful later in the paper, so
we will go through it in detail. This will be done in the next section.

\section{The dimensions of the $U^+_q$ homogeneous components}

\noindent In this section we do two things. First we
compute the dimension of each homogeneous component
for the $\mathbb N^2$-grading of $U^+_q$.
Then we use this data to
characterize a certain type of PBW basis for $U^+_q$;
this characterization will be invoked later in the paper to
obtain the alternating PBW basis.

\begin{definition}\rm
Let the set $\mathcal R$ consist of the ordered pairs
 $(r,s) \in \mathbb N^2$ such that 
$|r-s|\leq 1$ and $(r,s) \not=(0,0)$.
\end{definition}

\begin{definition}\rm Define a generating function
in two commuting indeterminates $\lambda,\mu$:
\begin{align*}
\Phi(\lambda,\mu) =\prod_{(r,s) \in \mathcal R} \frac{1}{1-\lambda^r \mu^s}
\end{align*}
In more detail,
\begin{align*}
\Phi(\lambda,\mu) = 
\prod_{\ell=1}^\infty 
\frac{1}{1-\lambda^{\ell} \mu^{\ell-1}}\,
\frac{1}{1-\lambda^{\ell} \mu^{\ell}}\,
\frac{1}{1-\lambda^{\ell-1} \mu^{\ell}}.
\end{align*}
\end{definition}

\begin{definition}
\label{def:dij}
\rm 
For $(i,j) \in \mathbb N^2$ let $d_{i,j}$ denote the coefficient
of $\lambda^i \mu^j$ in $\Phi(\lambda,\mu)$. Thus
\begin{align*}
\Phi(\lambda,\mu) &= 
\prod_{(r,s) \in \mathcal R} (1+ \lambda^r \mu^s + \lambda^{2r} 
\mu^{2s} + \cdots )
\\
& = \sum_{(i,j) \in \mathbb N^2} d_{i,j} \lambda^i \mu^j.
\end{align*}
Note that $d_{i,j} \in \mathbb N$ for $(i,j) \in \mathbb N^2$.
\end{definition}

\begin{example} 
\label{ex:tabledij}
\rm
For $0 \leq i,j\leq 6$ the number $d_{i,j}$
is given in the $(i,j)$-entry of the matrix below:
\begin{align*}
	\left(
         \begin{array}{ccccccc}
               1&1&1&1&1&1&1 
	       \\
               1&2&3&3&3&3&3 
               \\
	       1&3&6&8&9&9&9 
              \\
	      1&3&8&14&19&21&22 
             \\
	     1&3&9&19&32&42&48 
            \\
	    1&3&9&21&42&66&87
	    \\
	    1&3&9&22&48&87&134
                  \end{array}
\right)
\end{align*}
\end{example}

\begin{definition}
\label{def:omega}
\rm A subset  $\Omega \subseteq U^+_q$
is called {\it feasible} whenever
\begin{enumerate}
\item[\rm (i)] each element of $\Omega$ is homogeneous 
with respect to the $\mathbb N^2$-grading of $U^+_q$;
\item[\rm (ii)] there is a bijection $\Omega \to \mathcal R$
that sends each element of $\Omega$ to its degree.
\end{enumerate}
\end{definition}

\noindent
Let $\Omega$ denote a feasible subset of $U^+_q$ and
let $<$ denote a linear order on 
$\Omega$.  Consider the following vectors in $U^+_q$:
    \begin{align}
    a_1 a_2 \cdots a_n \qquad \quad  n\in \mathbb N, \quad \qquad
    a_1,a_2,\ldots, a_n \in \Omega, 
    \qquad \quad  a_1 \leq a_2 \leq \cdots \leq a_n.
 \label{eq:set} 
  \end{align}
Note that each vector in 
 (\ref{eq:set}) is homogeneous 
with respect to the $\mathbb N^2$-grading of $U^+_q$.
\begin{lemma}
\label{lem:dijmeaning}
With the above notation,
 for $(i,j)\in \mathbb N^2$ 
 the number of vectors in 
{\rm (\ref{eq:set})}
that have degree $(i,j)$ is equal to 
the integer $d_{i,j}$ from 
Definition \ref{def:dij}.
\end{lemma}
\begin{proof} For $a\in \Omega$ with degree $(r,s)$
the contribution of $a$ to $\Phi(\lambda,\mu)$ is
\begin{align*}
\frac{1}{1-\lambda^r \mu^s} = 
1+ \lambda^r \mu^s + \lambda^{2r} \mu^{2s} + \cdots.
\end{align*}
\end{proof}

\begin{corollary}
\label{cor:dimUq}
For $(i,j)\in \mathbb N^2$ the $(i,j)$-homogeneous
component of $U^+_q$ has dimension $d_{i,j}$.
\end{corollary}
\begin{proof} Let the set $\Omega$ 
consist of the elements
   (\ref{eqUq:PBWintro}). 
The set $\Omega$ is feasible
by
Definition
\ref{def:omega}
and the table below Proposition
    \ref{prop:PBWbasis}.
Endow $\Omega$ with the linear order $<$ from
    Proposition \ref{prop:PBWbasis}.
By Definition
\ref{def:PBW} and
Proposition
   \ref{prop:PBWbasis},
the vectors
 (\ref{eq:set}) form a basis for the vector space $U^+_q$.
We mentioned earlier
  that every vector in 
 (\ref{eq:set}) is homogeneous with respect to the
 $\mathbb N^2$-grading of $U^+_q$.
So for $(i,j)\in \mathbb N^2$ the set of
vectors in 
 (\ref{eq:set}) that have degree $(i,j)$ is a basis
for the $(i,j)$-homogeneous component of  
 $U^+_q$.
The result follows in view of Lemma
\ref{lem:dijmeaning}.
\end{proof}

\begin{lemma}
\label{lem:pre}
Let $\Omega$ denote a feasible subset of
$U^+_q$ and let $<$ denote a linear order on $\Omega$.
Then for $(i,j) \in \mathbb N^2$
the following {\rm (i)--(iii)} are equivalent:
\begin{enumerate}
\item[\rm (i)] the vectors in
 {\rm (\ref{eq:set})} that have degree $(i,j)$ span
 the $(i,j)$-homogeneous component of $U^+_q$;
\item[\rm (ii)] the vectors in
{\rm (\ref{eq:set})} that have degree $(i,j)$ are linearly
 independent;
\item[\rm (iii)] the vectors in
 {\rm (\ref{eq:set})} that have degree $(i,j)$ form a basis for 
 the $(i,j)$-homogeneous component of $U^+_q$.
\end{enumerate}
 \end{lemma}
\begin{proof} By
Lemma
\ref{lem:dijmeaning}, Corollary
\ref{cor:dimUq} and the linear algebra of finite-dimensional
vector spaces.
\end{proof}

\begin{proposition}
\label{prop:useful}
Let $\Omega$ denote a feasible subset of
$U^+_q$ and let $<$ denote a linear order on $\Omega$.
Then the following {\rm (i)--(v)} are equivalent:
\begin{enumerate}
\item[\rm (i)] the equivalent conditions of
Lemma
\ref{lem:pre} hold for all $(i,j) \in \mathbb N^2$;
\item[\rm (ii)] 
the vectors 
 {\rm (\ref{eq:set})} span
$U^+_q$;
\item[\rm (iii)] the vectors 
 {\rm (\ref{eq:set})} are linearly independent;
\item[\rm (iv)] the vectors 
 {\rm (\ref{eq:set})} form a basis for $U^+_q$;
\item[\rm (v)] $\Omega$ in order $<$ forms 
a PBW basis for $U^+_q$.
\end{enumerate}
 \end{proposition}
\begin{proof} Condition (i) is equivalent to each of
(ii), (iii), (iv) since 
 $U^+_q$ is a direct sum
of its homogeneous components.
Conditions (iv), (v) are equivalent
by 
Definition
   \ref{prop:PBWbasis}.
\end{proof}
\noindent In Section 10 we will use Proposition
\ref{prop:useful} to obtain the alternating PBW basis for $U^+_q$.

\section{Embedding $U^+_q$ into a $q$-shuffle algebra}

\noindent In this section we recall an  embedding, due to 
Rosso \cite{rosso1, rosso},
of $U^+_q$ into a $q$-shuffle algebra.
For this $q$-shuffle algebra the underlying vector space is
a free algebra on two generators. We begin by describing this
free algebra.
\medskip

\noindent 
 Let $x,y$
denote noncommuting indeterminates, and let $\mathbb V$ denote the free
algebra with generators $x$, $y$.
By a {\it letter} in $\mathbb V$ we mean $x$ or $y$.
For $n \in \mathbb N$, a {\it word of length $n$} in $\mathbb V$
is a product of letters $v_1 v_2 \cdots v_n$.
 We interpret the word of length zero
 to be the multiplicative
 identity in $\mathbb V$; this word is called {\it trivial} and denoted by $1$.
 The vector space $\mathbb V$ has a basis consisting
 of its words; this basis is called {\it standard}.
\medskip

\noindent We mention some symmetries of the free algebra $\mathbb V$.
\begin{lemma}
\label{lem:Vsym}
There exists a unique automorphism $\sigma$ of the free algebra
$\mathbb V$ that swaps $x,y$. There exists a unique antiautomorphism
$S$ of the free algebra $\mathbb V$ that fixes each of $x,y$.
\end{lemma}

\noindent The free algebra $\mathbb V$ has a $\mathbb N^2$-grading
for which $x$ and $y$ are homogeneous, with degrees
$(1,0)$ and $(0,1)$  respectively. For 
 $(i,j)\in \mathbb N^2$ 
let $\mathbb V_{i,j}$ denote the $(i,j)$-homogeneous component.
These homogeneous components are described as follows.
Let $w=v_1v_2\cdots v_n$ denote a word in $\mathbb V$.
The {\it $x$-degree} of $w$ is the
cardinality of the set
$\lbrace i | 1 \leq i \leq n,\;v_i = x\rbrace$.
The $y$-degree of $w$ is similarly defined.
For $(i,j)\in \mathbb N^2$ 
the subspace $\mathbb V_{i,j}$ has a basis consisting
of the words in $\mathbb V$ that have $x$-degree $i$ and
$y$-degree $j$. The dimension of $\mathbb V_{i,j}$ is equal to
the binomial coefficient $\binom{i+j}{i}$.
\medskip

\noindent We have been discussing the free algebra $\mathbb V$. There is
another algebra structure on $\mathbb V$,
called the $q$-shuffle algebra.
This algebra was introduced by Rosso
\cite{rosso1, rosso} and described further by Green
\cite{green}. We will adopt the approach of
\cite{green}, which is well suited to our purpose.
The $q$-shuffle product
is denoted by $\star$. To describe this product, we start
with some special cases.
We have $1 \star v = v \star 1 = v$ for $v \in \mathbb V$.
 For
letters $u,v$ we have
\begin{align*}
u \star v = uv + vu q^{\langle u,v\rangle}
\end{align*}
\noindent where
\bigskip
\centerline{
\begin{tabular}[t]{c|cc}
$\langle\,,\,\rangle$ & $x$ & $y$
   \\  \hline
   $x$ &
   $2$ & $-2$
     \\
     $y$ &
      $-2$ & $2$
         \\
              \end{tabular}
              }
              \medskip

\noindent
  So 
   \begin{align}
    &x \star y = xy+ q^{-2}yx,
 \qquad \qquad \quad
   y \star x = yx + q^{-2}xy,
\label{eq:xyprod}
\\
 &x \star x = (1+q^2)xx,
 \qquad \qquad \quad
    y\star y = (1+q^2)yy.
\label{eq:xxprod}
\end{align}
\noindent For a letter
 $u$ and
  a nontrivial word $v= v_1v_2\cdots v_n$,
  \begin{align}
  \label{eq:Xcv}
  &u \star v =
  \sum_{i=0}^n v_1 \cdots v_{i} u v_{i+1} \cdots v_n
  q^{
  \langle v_1, u\rangle+
  \langle v_2, u\rangle+
  \cdots + \langle v_{i}, u\rangle},
  \\
  &v \star u = \sum_{i=0}^n v_1 \cdots v_{i} u v_{i+1} \cdots v_n
  q^{
  \langle  v_{n},u\rangle
  +
  \langle v_{n-1},u\rangle
  +
  \cdots
  +
  \langle v_{i+1},u\rangle
  }.
  \label{eq:vcX}
  \end{align}
\noindent
   For example
   \begin{align*}
& x\star (yyy)= xyyy+ q^{-2} yxyy+ q^{-4} yyxy+q^{-6}yyyx,
    \\
   & (xyx) \star y = xyxy +
                (1+q^{-2})xyyx +
                  q^{-2} yxyx.
   \end{align*}

\noindent For nontrivial words $u=u_1u_2\cdots u_r$
and $v=v_1v_2\cdots v_s$ in $\mathbb V$,
\begin{align}
\label{eq:uvcirc}
&u \star v  = u_1\bigl((u_2\cdots u_r) \star v\bigr)
+ v_1\bigl(u \star (v_2 \cdots v_s)\bigr)
q^{
\langle u_1, v_1\rangle +
\langle u_2, v_1\rangle +
\cdots
+
\langle u_r, v_1\rangle},
\\
\label{eq:uvcirc2}
&u\star v =
\bigl(u \star (v_1 \cdots v_{s-1})\bigr)v_s +
\bigl((u_1 \cdots u_{r-1}) \star v\bigr)u_r
q^{
\langle u_r, v_1\rangle +
\langle u_r, v_2\rangle + \cdots +
\langle u_r, v_s\rangle
}.
\end{align}
 For example
   \begin{align*}
   (xx)\star (yyy) =&
            xxyyy  +
            q^{-2}xyxyy +
            q^{-4}xyyxy +
            q^{-6}xyyyx +
             q^{-4}yxxyy
             \\
            &+ q^{-6}yxyxy +
             q^{-8}yxyyx +
             q^{-8}yyxxy +
             q^{-10}yyxyx+
             q^{-12}yyyxx,
       \\
       (xy)\star (xxyy) = & xyxxyy+ xxyyxy+ \lbrack 2 \rbrack_q^2 xxyxyy+
        \lbrack 3 \rbrack_q^2 xxxyyy.
   \end{align*}
\noindent 
The map $\sigma$ from Lemma
\ref{lem:Vsym}
is an automorphism of the $q$-shuffle algebra
$\mathbb V$. The map $S$ from
Lemma
\ref{lem:Vsym}
is an antiautomorphism of the $q$-shuffle algebra
$\mathbb V$. Below Lemma
\ref{lem:Vsym} we mentioned an $\mathbb N^2$-grading for the
free algebra $\mathbb V$. This is also an $\mathbb N^2$-grading
for the $q$-shuffle algebra $\mathbb V$.
\medskip

\begin{definition}
\label{def:Usub}
\rm Let $U$ denote the subalgebra of the
$q$-shuffle algebra $\mathbb V$ that is generated by $x,y$.
\end{definition}
\noindent Shortly we will see that $U \not=\mathbb V$.
\medskip

\noindent 
The algebra $U$ is described as follows.
With some work
(or by
\cite[Theorem~13]{rosso1},
\cite[p.~10]{green})
one
obtains
\begin{align}
&
x \star x \star x \star y -
\lbrack 3 \rbrack_q
x \star x\star y \star x +
\lbrack 3 \rbrack_q
x \star y \star x \star x -
y \star x \star x \star x  = 0,
\label{eq:qsc1}
\\
&
y \star y \star y \star x -
\lbrack 3 \rbrack_q
y \star y \star x \star y +
\lbrack 3 \rbrack_q
y \star x \star y \star y -
x \star y \star y \star y  = 0.
\label{eq:qsc2}
\end{align}
So in the $q$-shuffle algebra $\mathbb V$ the elements
$x,y$ satisfy the
$q$-Serre relations.
Consequently there exists an algebra homomorphism
$\natural$ from $U^+_q$ to the $q$-shuffle algebra $\mathbb V$,
that sends $A\mapsto x$ and $B\mapsto y$.
The map $\natural$ has image $U$
by Definition
\ref{def:Usub}, 
and is injective by
  \cite[Theorem~15]{rosso}. 
Therefore $\natural: U^+_q \to U$ is an algebra isomorphism.
By construction  the following diagrams commute:

\begin{equation*}
{\begin{CD}
U^+_q @>\natural  >>
               {\mathbb V}
              \\
         @V \sigma VV                   @VV \sigma V \\
         U^+_q @>>\natural >
                                  {\mathbb  V}
                        \end{CD}} \qquad \qquad \qquad
{\begin{CD}
U^+_q @>\natural  >>
               {\mathbb V}
              \\
         @V S VV                   @VV S V \\
         U^+_q @>>\natural >
                                  {\mathbb  V}
                        \end{CD}} \qquad \qquad
   \end{equation*}
\noindent 
Earlier we mentioned an $\mathbb N^2$-grading for 
both the algebra $U^+_q$ and the $q$-shuffle algebra $\mathbb V$.
These gradings are related as follows.
The algebra $U$ has an $\mathbb N^2$-grading inherited from
$U^+_q$ via $\natural$. With respect to this grading,
for
$(i,j)\in \mathbb N^2$ the $(i,j)$-homogeneous component of
$U$ is the 
$\natural$-image of the $(i,j)$-homogeneous component
of $U^+_q$. This homogeneous component is equal to
$\mathbb V_{i,j} \cap U$.
\medskip

\noindent 
Next we show that $U \not=\mathbb V$.
For $(i,j) \in \mathbb N^2 $,  the dimension of 
 $\mathbb V_{i,j}\cap U$
is $d_{i,j}$ and the dimension of
$\mathbb V_{i,j}$
is $\binom{i+j}{i}$.
There exists $(i,j)\in \mathbb N^2$ such that
$d_{i,j}<\binom{i+j}{i}$.  Therefore $U\not=\mathbb V$.
\medskip

\noindent Next we describe how the map $\natural $ acts on
the PBW basis elements 
   (\ref{eqUq:PBWintro}).
Define $\overline x = 1$
and $\overline y = -1$.
A word $v_1v_2\cdots v_n$ in $\mathbb V$
is {\it Catalan} whenever
$\overline v_1+
\overline v_2+\cdots  + 
\overline v_i$ is nonnegative for
$1 \leq i \leq n-1$ and zero for $i=n$.
In this case $n$ is even.
For $n\geq 0$ define
\begin{align*}
C_n =
  \sum v_1v_2\cdots v_{2n}
\lbrack 1\rbrack_q
\lbrack 1+\overline v_1\rbrack_q
\lbrack 1+\overline v_1+\overline v_2\rbrack_q
\cdots
\lbrack 1+\overline v_1+\overline v_2+ \cdots +\overline v_{2n}\rbrack_q,
\end{align*}
where the sum is over all the Catalan words $v_1 v_2 \cdots v_{2n}$
in $\mathbb V$ that have length $2n$. For example
\begin{align*}
C_0 = 1, \qquad \qquad 
C_1 = \lbrack 2 \rbrack_q xy, \qquad \qquad 
C_2 =
\lbrack 2 \rbrack^2_q xyxy+
\lbrack 3 \rbrack_q \lbrack 2 \rbrack^2_q xxyy.
\end{align*}
By \cite[Theorem~1.7]{catalan} the map
$\natural$ sends
\begin{align*}
E_{n\delta+\alpha_0} \mapsto  q^{-2n}(q-q^{-1})^{2n} xC_n,
\qquad \qquad 
E_{n\delta+\alpha_1} \mapsto
 q^{-2n}(q-q^{-1})^{2n} C_ny
\end{align*}
 for $n\geq 0$, and
\begin{align*}
E_{n\delta} \mapsto  -q^{-2n}(q-q^{-1})^{2n-1} C_n
\end{align*}
 for $n\geq 1$.  See
\cite[p.~696]{leclerc} for more information about $\natural$.

\section{The alternating words}

In this section we introduce a type of word in $\mathbb V$,
said to be alternating. 
We display some relations in
 the $q$-shuffle algebra
$\mathbb V$, that are satisfied by the alternating words. 
Later in the paper we will use these relations to 
show that the alternating words are contained in $U$.

\begin{definition}\rm
A word
$v_1v_2\cdots v_n$ in $\mathbb V$ is called {\it alternating}
whenever $n\geq 1$ and 
$v_{i-1} \not=v_i$ for $2 \leq i \leq n$.
Thus an alternating word has the form $\cdots xyxy\cdots$.
\end{definition}

\begin{definition} 
\label{def:WWGG}
\rm We name the alternating words as follows:
\begin{align}
&W_0 = x, \qquad W_{-1} = xyx, \qquad W_{-2} = xyxyx, \qquad \ldots
\label{eq:Wm}
\\
&W_1 = y, \qquad W_{2} = yxy, \qquad W_{3} = yxyxy, \qquad \ldots 
\label{eq:Wp}
\\
&G_{1} = yx, \qquad G_{2} = yxyx,  \qquad G_3 = yxyxyx, \qquad \ldots 
\label{eq:G}
\\
&\tilde G_{1} = xy, \qquad \tilde G_{2} =
xyxy,\qquad \tilde G_3 = xyxyxy, \qquad \ldots
\label{eq:Gt}
\end{align}
For notational convenience define $G_0=1$ and $\tilde G_0=1$.
So for $k \in \mathbb N$,
\bigskip

\centerline{
\begin{tabular}[t]{c|c|c|c|c}
   {\rm name} & {\rm description} & {\rm $x$-degree} & 
   {\rm $y$-degree} & 
   {\rm length}
   \\
   \hline
$W_{-k}$ & $xyxy\cdots x$ & $k+1$ & $k$ & $2k+1$
\\
$W_{k+1}$ & $yxyx\cdots y$ & $k$ & $k+1$ & $2k+1$
\\
$G_{k}$ & $yxyx\cdots x$ & $k$ & $k$ & $2k$
\\
$\tilde G_{k}$ & $xyxy\cdots y$ & $k$ & $k$ & $2k$
\end{tabular}
}
\bigskip
\end{definition}

\begin{lemma}
\label{lem:sigSact}
The maps $\sigma$, $S$  from
Lemma \ref{lem:Vsym}
act on the alternating
words as follows. For $k \in \mathbb N$,
\begin{enumerate}
\item[\rm(i)]
the map $\sigma $ sends
\begin{align*}
W_{-k} \mapsto W_{k+1}, \qquad \quad
W_{k+1} \mapsto W_{-k}, \qquad \quad
G_{k} \mapsto \tilde G_{k}, \qquad \quad
\tilde G_{k} \mapsto G_{k};
\end{align*}
\item[\rm (ii)] the map $S$ sends
\begin{align*}
W_{-k} \mapsto W_{-k}, \qquad \quad
W_{k+1} \mapsto W_{k+1}, \qquad \quad
G_{k} \mapsto \tilde G_{k}, \qquad \quad
\tilde G_{k} \mapsto G_{k}.
\end{align*}
\end{enumerate}
\end{lemma}
\begin{proof} By Lemma
\ref{lem:Vsym}
and Definition
\ref{def:WWGG}.
\end{proof}

\begin{lemma} 
\label{lem:free} 
For $k\in \mathbb N$ the following holds
in the free algebra $\mathbb V$:
\begin{align*}
&W_{-k} = xG_k = \tilde G_kx,
\qquad \qquad 
G_{k+1} = y W_{-k} =  W_{k+1} x,
\\
&
W_{k+1} = y\tilde G_k =  G_k y,
\qquad \qquad 
\tilde G_{k+1} = x W_{k+1} =  W_{-k} y.
\end{align*}
\end{lemma}
\begin{proof} Use Definition
\ref{def:WWGG}.
\end{proof}

\noindent We are going to show that $U$ contains
every alternating word. As a warmup, consider 
the alternating words
$xy$ and $yx$.
Using
(\ref{eq:xyprod}),
\begin{align}
xy = q \frac{qx\star y - q^{-1} y\star x}{q^2-q^{-2}},
\qquad \qquad
yx = q \frac{q y\star x - q^{-1} x\star y}{q^2-q^{-2}}.
\label{eq:xypoly}
\end{align}
\noindent Therefore $U$ contains $xy$ and $yx$.
In order to handle longer alternating words,
we will develop some relations involving the $q$-shuffle
product. Next we describe the $q$-shuffle product of
a letter and an alternating word.

\begin{lemma} 
\label{lem:stardata1}
For $k \in \mathbb N$,
\begin{align}
x \star W_{-k} &= (1+q^2)\sum_{i=0}^k \tilde G_i x^2 G_{k-i},
\label{eq:xWm}
\\
x \star W_{k+1} &=
q^{-2} G_{k+1}
+
\tilde G_{k+1}
+ 
           (1+q^{-2})\sum_{i=0}^{k-1}  W_{i+1} x^2 W_{k-i},
\label{eq:xWp}
\\
x \star G_{k} &= W_{-k} +
           (1+q^{-2})\sum_{i=0}^{k-1} W_{i+1} x^2 G_{k-i-1},
\label{eq:xG}
\\
x \star {\tilde G}_{k} &= W_{-k} +
           (1+q^{2})\sum_{i=0}^{k-1} \tilde G_{k-i-1} x^2 W_{i+1}
\label{eq:xGt}
\end{align}
and
\begin{align}
y \star W_{-k} &= G_{k+1} + q^{-2} \tilde G_{k+1} + 
           (1+q^{-2})\sum_{i=0}^{k-1} W_{-i} y^2 W_{i-k+1},
\\
y \star W_{k+1} &= (1+q^2)\sum_{i=0}^k G_i y^2 \tilde G_{k-i},
\\
y \star G_{k} &= W_{k+1} +
           (1+q^{2})\sum_{i=0}^{k-1} G_{k-i-1} y^2 W_{-i},
\\
y \star \tilde G_{k} &= W_{k+1} +
           (1+q^{-2})\sum_{i=0}^{k-1} W_{-i} y^2 \tilde G_{k-i-1}
\end{align}
\end{lemma}
\begin{proof} 
  Use (\ref{eq:Xcv})
and Definition
\ref{def:WWGG}.
\end{proof}

\noindent Next we describe the $q$-shuffle product of an alternating 
word and a letter.

\begin{lemma}
\label{lem:stardata2}
For $k \in \mathbb N$,
\begin{align}
 W_{-k} \star x &= (1+q^2)\sum_{i=0}^k 
 {\tilde G}_{k-i} x^2 G_{i},
\\
W_{k+1} \star x &=
 G_{k+1}
+
q^{-2} {\tilde G}_{k+1}
+ 
           (1+q^{-2})\sum_{i=0}^{k-1} 
	   W_{k-i} x^2 W_{i+1},
\\
 G_{k} \star x &= W_{-k} +
           (1+q^{2})\sum_{i=0}^{k-1} 
	   W_{i+1} x^2 G_{k-i-1},
\\
{\tilde G}_{k} \star x &= W_{-k} +
           (1+q^{-2})\sum_{i=0}^{k-1}
	  {\tilde  G}_{k-i-1} x^2 W_{i+1}
\end{align}
and
\begin{align}
 W_{-k} \star y&=
 q^{-2}  G_{k+1} + 
 {\tilde G}_{k+1}+
           (1+q^{-2})\sum_{i=0}^{k-1}
	   W_{i-k+1} y^2 W_{-i},
\\
 W_{k+1} \star y&= (1+q^2)\sum_{i=0}^k 
G_{k-i} y^2 {\tilde G}_{i},
\\
 G_{k} \star y &= W_{k+1} +
           (1+q^{-2})\sum_{i=0}^{k-1}
	   G_{k-i-1} y^2 W_{-i},
	   \\
{\tilde G}_{k} \star y&= W_{k+1} +
           (1+q^{2})\sum_{i=0}^{k-1}
	   W_{-i} y^2 {\tilde G}_{k-i-1}.
\end{align}
\end{lemma}
  \begin{proof}
Use (\ref{eq:vcX}) 
and Definition
\ref{def:WWGG}.
\end{proof}

\begin{proposition}
\label{prop:rel1}
For $k \in \mathbb N$
the following holds in the $q$-shuffle algebra $\mathbb V$:
\begin{align}
&
 \lbrack  W_0,  W_{k+1}\rbrack= 
\lbrack  W_{-k},  W_{1}\rbrack=
(1-q^{-2})({\tilde G}_{k+1} -  G_{k+1}),
\label{eq:3p1vv}
\\
&
\lbrack  W_0,  G_{k+1}\rbrack_q= 
\lbrack {{\tilde G}}_{k+1},  W_{0}\rbrack_q= 
 (q-q^{-1})W_{-k-1},
\label{eq:3p2vv}
\\
&
\lbrack G_{k+1},  W_{1}\rbrack_q= 
\lbrack  W_{1}, { {\tilde G}}_{k+1}\rbrack_q= 
(q-q^{-1}) W_{k+2}.
\label{eq:3p3vv}
\end{align}
\end{proposition}
\begin{proof} These relations are routinely checked using
Lemmas 
\ref{lem:stardata1},
\ref{lem:stardata2}.
\end{proof}

\begin{note}\rm
We have a comment about notation.
In 
Proposition
\ref{prop:rel1} we used the commutator and $q$-commutator notation.
Throughout the 
paper, for any equation
in the $q$-shuffle algebra $\mathbb V$ that involves a commutator
or $q$-commutator, it is understand that these commutators
are computed using the $q$-shuffle product $\star$.
\end{note}

\noindent We just displayed some relations for the
alternating words in $\mathbb V$. Shortly
we will display some more general relations for the
alternating words. To obtain these relations
we use the following identities.

\begin{lemma}
\label{lem:prep}
For $k,\ell \in \mathbb N$,
\begin{align*}
W_{-k} \star W_{-\ell} &=  x(G_k \star W_{-\ell}) + x(W_{-k}\star G_\ell)q^2
= q^2(\tilde G_k \star W_{-\ell})x + (W_{-k} \star \tilde G_\ell)x,
\\
W_{-k} \star W_{\ell+1} &=  x(G_k \star W_{\ell+1}) + y(W_{-k}\star \tilde G_\ell)q^{-2}
= q^{-2}(\tilde G_k \star W_{\ell+1})x + (W_{-k} \star  G_\ell)y,
\\
W_{k+1} \star W_{-\ell} &=  y(\tilde G_k \star W_{-\ell}) + x(W_{k+1}\star G_\ell)q^{-2}
= q^{-2}( G_k \star W_{-\ell})y + (W_{k+1} \star \tilde G_\ell)x,
\\
W_{k+1} \star W_{\ell+1} &=  y(\tilde G_k \star W_{\ell+1}) + y(W_{k+1}\star \tilde  G_\ell)q^{2}
= q^{2}( G_k \star W_{\ell+1})y + (W_{k+1} \star  G_\ell)y
\end{align*}
\noindent and
\begin{align*}
W_{-k} \star G_{\ell+1} &= x(G_k \star G_{\ell+1}) + y(W_{-k}\star W_{-\ell})q^{-2}
=  (\tilde G_k \star G_{\ell+1})x + (W_{-k}\star W_{\ell+1})x,
\\
W_{-k} \star \tilde G_{\ell+1} &= x(G_k \star \tilde G_{\ell+1}) + 
x(W_{-k}\star W_{\ell+1})q^{2}
=  (\tilde G_k \star \tilde G_{\ell+1})x + (W_{-k}\star W_{-\ell})y,
\\
W_{k+1} \star G_{\ell+1} &= y(\tilde G_k \star G_{\ell+1}) + y(W_{k+1}\star W_{-\ell})q^{2}
=  (G_k \star G_{\ell+1})y + (W_{k+1}\star W_{\ell+1})x,
\\
W_{k+1} \star \tilde G_{\ell+1} &= y(\tilde G_k \star \tilde G_{\ell+1}) +
x(W_{k+1}\star W_{\ell+1})q^{-2}
=  ( G_k \star \tilde G_{\ell+1})y + (W_{k+1}\star W_{-\ell})y
\end{align*}
\noindent and
\begin{align*}
G_{k+1} \star W_{-\ell} &= y(W_{-k} \star W_{-\ell})+ x(G_{k+1} \star G_\ell) 
= q^2(W_{k+1} \star W_{-\ell})x + (G_{k+1} \star \tilde G_\ell)x,
\\
G_{k+1} \star W_{\ell+1} &= y(W_{-k} \star W_{\ell+1})+ y(G_{k+1} \star \tilde G_\ell) 
= q^{-2}(W_{k+1} \star W_{\ell+1})x + (G_{k+1} \star  G_\ell)y,
\\
\tilde G_{k+1} \star W_{-\ell} &= x(W_{k+1} \star W_{-\ell})+ x(\tilde G_{k+1} \star G_\ell) 
= q^{-2}(W_{-k} \star W_{-\ell})y + (\tilde G_{k+1} \star \tilde G_\ell)x,
\\
\tilde G_{k+1} \star W_{\ell+1} &= x(W_{k+1} \star W_{\ell+1})+ y(\tilde G_{k+1} \star \tilde G_\ell) 
= q^{2}(W_{-k} \star W_{\ell+1})y + (\tilde G_{k+1} \star  G_\ell)y
\end{align*}
\noindent and
\begin{align*}
G_{k+1} \star G_{\ell+1} &= y(W_{-k} \star G_{\ell+1}) + y(G_{k+1} \star W_{-\ell})
= (W_{k+1} \star G_{\ell+1})x + (G_{k+1} \star W_{\ell+1})x,
\\
G_{k+1} \star \tilde G_{\ell+1} &= y(W_{-k} \star \tilde G_{\ell+1}) + 
x(G_{k+1} \star W_{\ell+1})
= (W_{k+1} \star \tilde G_{\ell+1})x + (G_{k+1} \star W_{-\ell})y,
\\
\tilde G_{k+1} \star G_{\ell+1} &= x(W_{k+1} \star G_{\ell+1}) + y(\tilde G_{k+1} \star W_{-\ell})
= (W_{-k} \star G_{\ell+1})y + (\tilde G_{k+1} \star W_{\ell+1})x,
\\
\tilde G_{k+1} \star \tilde G_{\ell+1} &= x(W_{k+1} \star \tilde G_{\ell+1}) + x(\tilde G_{k+1} \star W_{\ell+1})
= (W_{-k} \star \tilde G_{\ell+1})y + (\tilde G_{k+1} \star W_{-\ell})y.
\end{align*}
\end{lemma}
\begin{proof} Use 
(\ref{eq:uvcirc}), 
(\ref{eq:uvcirc2}).
\end{proof}

\begin{proposition}
\label{prop:rel2}
For $k, \ell \in \mathbb N$ the
following relations hold in the $q$-shuffle algebra $\mathbb V$:
\begin{align}
&
\lbrack  W_{-k},  W_{-\ell}\rbrack=0,  \qquad 
\lbrack  W_{k+1},  W_{\ell+1}\rbrack= 0,
\label{eq:3p4vv}
\\
&
\lbrack  W_{-k},  W_{\ell+1}\rbrack+
\lbrack W_{k+1},  W_{-\ell}\rbrack= 0,
\label{eq:3p5vv}
\\
&
\lbrack  W_{-k},  G_{\ell+1}\rbrack+
\lbrack G_{k+1},  W_{-\ell}\rbrack= 0,
\label{eq:3p6vv}
\\
&
\lbrack W_{-k},  {\tilde G}_{\ell+1}\rbrack+
\lbrack  {\tilde G}_{k+1},  W_{-\ell}\rbrack= 0,
\label{eq:3p7vv}
\\
&
\lbrack  W_{k+1},  G_{\ell+1}\rbrack+
\lbrack   G_{k+1}, W_{\ell+1}\rbrack= 0,
\label{eq:3p8vv}
\\
&
\lbrack  W_{k+1},  {\tilde G}_{\ell+1}\rbrack+
\lbrack  {\tilde G}_{k+1},  W_{\ell+1}\rbrack= 0,
\label{eq:3p9vv}
\\
&
\lbrack  G_{k+1},  G_{\ell+1}\rbrack=0,
\qquad 
\lbrack {\tilde G}_{k+1},  {\tilde G}_{\ell+1}\rbrack= 0,
\label{eq:3p10vv}
\\
&
\lbrack {\tilde G}_{k+1},  G_{\ell+1}\rbrack+
\lbrack  G_{k+1},  {\tilde G}_{\ell+1}\rbrack= 0.
\label{eq:3p11v}
\end{align}
\end{proposition}
\begin{proof} Use Lemma
\ref{lem:prep} and induction on $k+\ell$.
\end{proof}

\begin{proposition}
\label{prop:rel3} For $k,\ell \in \mathbb N$ the following
relations hold in the $q$-shuffle algebra $\mathbb V$:
\begin{align}
&\lbrack W_{-k}, G_{\ell}\rbrack_q = 
\lbrack W_{-\ell}, G_{k}\rbrack_q,
\qquad \quad
\lbrack G_k, W_{\ell+1}\rbrack_q = 
\lbrack G_\ell, W_{k+1}\rbrack_q,
\label{eq:gg1}
\\
&
\lbrack \tilde G_k, W_{-\ell}\rbrack_q = 
\lbrack \tilde G_\ell, W_{-k}\rbrack_q,
\qquad \quad 
\lbrack W_{\ell+1}, \tilde G_{k}\rbrack_q = 
\lbrack W_{k+1}, \tilde G_{\ell}\rbrack_q,
\label{eq:gg2}
\\
&\lbrack G_{k}, \tilde G_{\ell+1}\rbrack -
\lbrack G_{\ell}, \tilde G_{k+1}\rbrack =
q\lbrack W_{-\ell}, W_{k+1}\rbrack_q-
q\lbrack W_{-k}, W_{\ell+1}\rbrack_q,
\label{eq:gg3}
\\
&\lbrack \tilde G_{k},  G_{\ell+1}\rbrack -
\lbrack \tilde G_{\ell},  G_{k+1}\rbrack =
q \lbrack W_{\ell+1}, W_{-k}\rbrack_q-
q\lbrack W_{k+1}, W_{-\ell}\rbrack_q,
\label{eq:gg4}
\\
&\lbrack G_{k+1}, \tilde G_{\ell+1}\rbrack_q -
\lbrack G_{\ell+1}, \tilde G_{k+1}\rbrack_q =
q\lbrack W_{-\ell}, W_{k+2}\rbrack-
q\lbrack W_{-k}, W_{\ell+2}\rbrack,
\label{eq:gg5}
\\
&\lbrack \tilde G_{k+1},  G_{\ell+1}\rbrack_q -
\lbrack \tilde G_{\ell+1},  G_{k+1}\rbrack_q =
q \lbrack W_{\ell+1}, W_{-k-1}\rbrack-
q\lbrack W_{k+1}, W_{-\ell-1}\rbrack.
\label{eq:gg6}
\end{align}
\end{proposition}
\begin{proof} Each equation is verified by evaluating
both sides using Lemma
\ref{lem:prep} and simplifying the result using
Proposition
\ref{prop:rel2}.
\end{proof}

\noindent We emphasize 
 a few points about Definition 
\ref{def:WWGG}.

\begin{lemma} Referring to Definition 
\ref{def:WWGG}, the following
{\rm (i)--(v)} hold.
\begin{enumerate}
\item[\rm (i)] For each of the lines
{\rm (\ref{eq:Wm})--(\ref{eq:Gt})}
the listed elements mutually commute.
\item[\rm (ii)]
An alternating word is listed in 
{\rm (\ref{eq:Wm})} if and only if it commutes with $x$.
\item[\rm (iii)]
An alternating word is listed in 
{\rm (\ref{eq:Wp})} if and only if it commutes with $y$.
\item[\rm (iv)]
An alternating word is listed in 
{\rm (\ref{eq:G})} if and only if it commutes with 
$q y\star x - q^{-1} x \star y$.
\item[\rm (v)]
An alternating word is listed in 
{\rm (\ref{eq:Gt})} if and only if it commutes with 
$q x \star y - q^{-1} y \star x$.
\end{enumerate}
\end{lemma}
\begin{proof} (i) By
(\ref{eq:3p4vv}),
(\ref{eq:3p10vv}).
\\
\noindent (ii) Use
(\ref{eq:xG}),
(\ref{eq:xGt}) and
(\ref{eq:3p1vv}),
(\ref{eq:3p2vv}).
\\
\noindent (iii) Use (ii) above and the $\sigma$ action from
Lemma \ref{lem:sigSact}.
\\
\noindent (iv)
Use
(\ref{eq:xypoly})  together with
(\ref{eq:3p6vv}),
(\ref{eq:3p8vv}),
(\ref{eq:gg4}) at $\ell=0$.
\\
\noindent (v) 
Use (iv) above and the $\sigma$ action from
Lemma \ref{lem:sigSact}.
\end{proof}

\section{Commutator relations for alternating words, part I}

\noindent Our next general goal is to
obtain a commutator relation for every pair of alternating 
words.
As we pursue this goal, it is convenient to split the argument
into two cases.
In the present (resp. next) section we treat the case in which the
pair of alternating words have length of opposite (resp. same) parity.
Throughout this section fix $n \in \mathbb N$. 
\medskip

\newpage
\noindent 
In Table 1 below, we see four cases.
For each case we have
a set $L$ and a set $R$:
\bigskip

\centerline{
\begin{tabular}[t]{c|cc}
   {\rm case} &  {\rm set $L$} & {\rm set $R$} 
   \\
   \hline
$1$ &$\lbrace  G_i \star W_{-j}|i,j\in \mathbb N, \;\; i+j=n\rbrace$
& $\lbrace W_{-j} \star G_i | i,j\in \mathbb N, \;\; i+j=n\rbrace$
\\
$2$ & $\lbrace G_i \star W_{j+1} |i,j\in \mathbb N, \;\; i+j=n\rbrace$
 &
$\lbrace  W_{j+1} \star G_i | i,j\in \mathbb N, \;\; i+j=n\rbrace$
\\
$3$ & $\lbrace \tilde G_i \star W_{-j}|i,j\in \mathbb N, \;\; i+j=n\rbrace$
   &
  $\lbrace W_{-j} \star \tilde G_i | i,j\in \mathbb N, \;\;
  i+j=n\rbrace$
\\
$4$& $\lbrace \tilde G_i \star W_{j+1} |i,j\in \mathbb N, \;\;  i+j=n\rbrace$
&
$\lbrace W_{j+1}\star \tilde G_i | i,j\in \mathbb N, \;\; i+j=n\rbrace$
\end{tabular}
}
\medskip

{
\centerline{Table 1}}
\medskip

\noindent In each case,  we will show that
the sets $L$  and $R$
have the same span.
To this end, order the set $L$
 as follows:
\bigskip

\centerline{
\begin{tabular}[t]{c|c}
   {\rm case} &  {\rm ordering of $L$} 
   \\
   \hline
$1$ &
$G_0 \star W_{-n}, \;\; G_n \star W_0, \;\;
G_1 \star W_{1-n}, \;\; G_{n-1} \star W_{-1}, \;\;
G_2 \star W_{2-n}, \;\; G_{n-2} \star W_{-2}, \;\; \ldots $
\\
$2$ &
$G_0 \star W_{n+1}, \;\; G_n \star W_1, \;\;
G_1 \star W_{n}, \;\; G_{n-1} \star W_{2}, \;\;
G_2 \star W_{n-1}, \;\; G_{n-2} \star W_{3}, \;\; \ldots
$
\\
$3$ &
$\tilde G_0 \star W_{-n}, \;\; \tilde G_n \star W_0, \;\;
\tilde G_1 \star W_{1-n}, \;\; \tilde G_{n-1} \star W_{-1}, \;\;
\tilde G_2 \star W_{2-n}, \;\; \tilde G_{n-2} \star W_{-2}, \;\; \ldots
$
\\
$4$ &
$\tilde G_0 \star W_{n+1}, \;\; \tilde G_n \star W_1, \;\;
\tilde G_1 \star W_{n}, \;\; \tilde G_{n-1} \star W_{2}, \;\;
\tilde G_2 \star W_{n-1}, \;\; \tilde G_{n-2} \star W_{3}, \;\; \ldots
$
\end{tabular}
}
\bigskip

\noindent Order the set $R$ as follows:
\bigskip

\centerline{
\begin{tabular}[t]{c|c}
   {\rm case} &  {\rm ordering of $R$} 
   \\
   \hline
$1$ &
$ W_{-n}\star G_0, \;\;  W_0 \star G_n, \;\;
W_{1-n} \star G_1, \;\;  W_{-1} \star G_{n-1}, \;\;
W_{2-n}\star G_2, \;\;   W_{-2} \star G_{n-2}, \;\; \ldots
$
\\
$2$ &
$ W_{n+1}\star G_0, \;\;  W_1 \star G_n, \;\;
W_{n} \star G_1, \;\;  W_{2} \star G_{n-1}, \;\;
W_{n-1}\star G_2, \;\;   W_{3} \star G_{n-2}, \;\; \ldots
$
\\
$3$ &
$ W_{-n}\star \tilde G_0, \;\;  W_0 \star \tilde G_n, \;\;
W_{1-n} \star \tilde G_1, \;\;  W_{-1} \star \tilde G_{n-1}, \;\;
W_{2-n}\star \tilde G_2, \;\;   W_{-2} \star \tilde G_{n-2}, \;\; \ldots
$
\\
$4$ &
$W_{n+1}\star \tilde G_0, \;\;  W_1 \star \tilde G_n, \;\;
W_{n} \star \tilde G_1, \;\;  W_{2} \star \tilde G_{n-1}, \;\;
W_{n-1}\star \tilde G_2, \;\;   W_{3} \star \tilde G_{n-2}, \;\; \ldots
$
\end{tabular}
}

\bigskip

\noindent 
Define sequences $\lbrace u_i\rbrace_{i=0}^n$,
$\lbrace v_i\rbrace_{i=0}^n$
as follows.
In cases 1 and 4 let 
$\lbrace u_i\rbrace_{i=0}^n$
(resp. $\lbrace v_i\rbrace_{i=0}^n$)
denote the given ordering of $L$ (resp. $R$).
In cases 2 and 3 let 
$\lbrace u_i\rbrace_{i=0}^n$
(resp. $\lbrace v_i\rbrace_{i=0}^n$)
denote the given ordering of $R$ (resp. $L$).


\begin{lemma}
\label{lem:transGW}
In each case 1--4 above,
the following holds for  $0 \leq j \leq n$.
\begin{enumerate} 
\item[\rm (i)] For $j$ even,
\begin{align*} 
&u_j = v_j + (1-q^{2})\sum_{i=0}^{j-1} (-1)^i v_i,
\qquad \qquad 
v_j = u_j + (1-q^{-2})\sum_{i=0}^{j-1} (-1)^i u_i.
\end{align*}
\item[\rm (ii)] For $j$ odd,
\begin{align*} 
&u_j = q^2 v_j + (1-q^{2})\sum_{i=0}^{j-1} (-1)^i v_i,
\qquad \qquad 
v_j = q^{-2} u_j + (1-q^{-2})\sum_{i=0}^{j-1} (-1)^i u_i.
\end{align*}
\end{enumerate}
\end{lemma}
\begin{proof} To obtain the result for case 1, use Propositions
\ref{prop:rel1},
\ref{prop:rel2},
\ref{prop:rel3} 
and induction on $j$.
To obtain the result for case 2 (resp. case 3) (resp. case 4),
apply $\sigma S$ (resp. $S$) (resp. $\sigma$) to everything
from case 1. 
\end{proof}

\noindent In Appendix A we present Lemma
\ref{lem:transGW} in an alternative form.

\begin{proposition}
\label{prop:table1}
For each case in Table 1,
 the sets $L$
and $R$ have the same span.
\end{proposition}
\begin{proof} By Lemma
\ref{lem:transGW}.
\end{proof}

\noindent  We mention some relations for later use.

\begin{proposition} 
\label{prop:attract}
For $n \in \mathbb N$,
\begin{align}
&
\sum_{k=0}^n  G_{n-k} \star W_{-k} q^{2k-n} = 
\sum_{k=0}^n W_{-k} \star  G_{n-k} q^{n-2k},
\label{eq:WGsum1}
\\
&
\sum_{k=0}^n  G_{n-k} \star W_{k+1} q^{n-2k} = 
\sum_{k=0}^n W_{k+1} \star  G_{n-k} q^{2k-n},
\label{eq:WGsum4}
\\
&
\sum_{k=0}^n \tilde G_{n-k} \star W_{-k} q^{n-2k} = 
\sum_{k=0}^n W_{-k} \star \tilde G_{n-k} q^{2k-n}.
\label{eq:WGsum3}
\\
&
\sum_{k=0}^n  \tilde G_{n-k} \star W_{k+1} q^{2k-n} = 
\sum_{k=0}^n W_{k+1} \star  \tilde G_{n-k} q^{n-2k},
\label{eq:WGsum2}
\end{align}
\end{proposition}
\begin{proof} 
To verify (\ref{eq:WGsum1}), evaluate each summand on
the left using case 1 of Lemma
\ref{lem:transGW}, and simplify the result.
To obtain 
(\ref{eq:WGsum4}) 
(resp.  (\ref{eq:WGsum3}))
(resp.  (\ref{eq:WGsum2}))
apply $\sigma S$ (resp. $S$) (resp. $\sigma$) to everything
in 
(\ref{eq:WGsum1}).
\end{proof}

\section{Commutator relations for alternating words, part II}

\noindent 
In this section we obtain a commutator relation for every 
pair of alternating words that have length of the same parity.
Throughout this section fix an integer $n\geq 1$.
\medskip

\noindent 
Consider the following sets:
\begin{align*}
N &=
\lbrace \tilde G_i \star  G_j |i,j\in \mathbb N, \;\; i+j=n\rbrace,
\\
S &=\lbrace  G_i \star \tilde G_j | i,j\in \mathbb N, \;\; i+j=n\rbrace,
\\
E&= \lbrace W_{i+1} \star W_{-j} |i,j\in \mathbb N, \;\; i+j=n-1\rbrace,
\\
W&= \lbrace W_{-i} \star  W_{j+1} |i,j\in \mathbb N, \;\; i+j=n-1\rbrace.
\end{align*}
\noindent We are going to show that the sets
\begin{align*}
N \cup E, \qquad 
N \cup W, \qquad 
S \cup E, \qquad 
S \cup W
\end{align*}
all have the same span. 
To this end, for each of $N,S,E,W$ we order the elements in two ways:
\newpage
\bigskip

\centerline{
\begin{tabular}[t]{c|c}
    &  {\rm ordering of $N$} 
   \\
   \hline
{\rm I}  &
$\tilde G_0 \star G_{n}, \;\; \tilde G_n \star G_0, \;\;
\tilde G_1 \star G_{n-1}, \;\; \tilde G_{n-1} \star G_{1}, \;\;
\tilde G_2 \star G_{n-2}, \;\; \tilde G_{n-2} \star G_{2}, \;\; \ldots $
\\
{\rm II} &
$\tilde G_n \star G_{0}, \;\; \tilde G_0 \star G_n, \;\;
\tilde G_{n-1} \star G_{1}, \;\; \tilde G_{1} \star G_{n-1}, \;\;
\tilde G_{n-2} \star G_{2}, \;\; \tilde G_{2} \star G_{n-2}, \;\; \ldots $
\end{tabular}
}
\bigskip

\centerline{
\begin{tabular}[t]{c|c}
    &  {\rm ordering of $S$} 
   \\
   \hline
{\rm I} &
$G_n \star \tilde G_{0}, \;\;  G_0 \star \tilde G_n, \;\;
 G_{n-1} \star \tilde G_{1}, \;\;  G_{1} \star \tilde G_{n-1}, \;\;
 G_{n-2} \star \tilde G_{2}, \;\;  G_{2} \star \tilde G_{n-2}, \;\; \ldots $
\\
{\rm II} &
$G_0 \star \tilde G_{n}, \;\;  G_n \star \tilde G_0, \;\;
 G_{1} \star \tilde G_{n-1}, \;\;  G_{n-1} \star \tilde G_{1}, \;\;
 G_{2} \star \tilde G_{n-2}, \;\;  G_{n-2} \star \tilde G_{2}, \;\; \ldots $
\end{tabular}
}
\bigskip

\centerline{
\begin{tabular}[t]{c|c}
    &  {\rm ordering of $E$} 
   \\
   \hline
{\rm I} &
$
W_n \star W_{0}, \;\;  W_1 \star  W_{1-n}, \;\;
W_{n-1} \star W_{-1}, \;\;  W_2 \star  W_{2-n}, \;\;
W_{n-2} \star W_{-2}, \;\;  W_3 \star  W_{3-n}, \;\; \ldots $
\\
{\rm II} &
$W_1 \star  W_{1-n}, \;\;  W_n \star  W_0, \;\;
 W_{2} \star  W_{2-n}, \;\;  W_{n-1} \star  W_{-1}, \;\;
 W_{3} \star  W_{3-n}, \;\;  W_{n-2} \star  W_{-2}, \;\; \ldots $
\end{tabular}
}
\bigskip

\centerline{
\begin{tabular}[t]{c|c}
    &  {\rm ordering of $W$} 
   \\
   \hline
{\rm I} &
$
W_0 \star W_{n}, \;\;  W_{1-n} \star  W_{1}, \;\;
W_{-1} \star W_{n-1}, \;\;  W_{2-n} \star  W_{2}, \;\;
W_{-2} \star W_{n-2}, \;\;  W_{3-n} \star  W_{3}, \;\; \ldots $
\\
{\rm II} &
$W_{1-n} \star  W_{1}, \;\;  W_0 \star  W_n, \;\;
 W_{2-n} \star  W_{2}, \;\;  W_{-1} \star  W_{n-1}, \;\;
 W_{3-n} \star  W_{3}, \;\;  W_{-2} \star  W_{n-2}, \;\; \ldots $
\end{tabular}
}
\bigskip

\noindent  Next we define sequences
$\lbrace u_i\rbrace_{i=0}^n$,
$\lbrace v_i\rbrace_{i=0}^n$,
$\lbrace U_i\rbrace_{i=0}^{n-1}$,
$\lbrace V_i\rbrace_{i=0}^{n-1}$ as follows. There are
four cases:
\medskip

\centerline{
\begin{tabular}[t]{c|cccc }
  {\rm case}  &  
  $\lbrace u_i\rbrace_{i=0}^n$  &
  $\lbrace v_i\rbrace_{i=0}^n$ &
  $\lbrace U_i\rbrace_{i=0}^{n-1}$ &
  $\lbrace V_i\rbrace_{i=0}^{n-1}$ 
   \\
   \hline
$1$ &
{\rm ordering I of $N$} &
{\rm ordering I of $S$} &
{\rm ordering I of $E$} &
{\rm ordering I of $W$} \\
$2$ &
{\rm ordering I of $S$} &
{\rm ordering I of $N$} &
{\rm ordering II of $E$} &
{\rm ordering II of $W$} \\
$3$ &
{\rm ordering II of $N$} &
{\rm ordering II of $S$} &
{\rm ordering I of $W$} &
{\rm ordering I of $E$} \\
$4$ &
{\rm ordering II of $S$} &
{\rm ordering II of $N$} &
{\rm ordering II of $W$} &
{\rm ordering II of $E$}
\end{tabular}
}
\bigskip

\begin{lemma}
\label{lem:ggww}
For each case 1--4 above we have the
following.
\begin{enumerate}
\item[\rm (i)] For $0 \leq j\leq n$ and $j$ even,
\begin{align*}
&u_j = v_j + (1-q^2)\sum_{i=0}^{j-1} (-1)^i V_i,
\qquad \qquad 
v_j = u_j - (1-q^2)\sum_{i=0}^{j-1} (-1)^i U_i.
\end{align*}
\item[\rm (ii)] For $0 \leq j \leq n$ and $j$ odd,
\begin{align*}
&u_j = v_j + (1-q^2)\sum_{i=0}^{j-2} (-1)^i V_i,
\qquad \qquad 
v_j = u_j - (1-q^2)\sum_{i=0}^{j-2} (-1)^i U_i.
\end{align*}
\item[\rm (iii)] For $0 \leq j \leq n-1$ and $j$ even,
\begin{align*}
&U_j = V_j + (1-q^{-2})\sum_{i=0}^{j+1} (-1)^i v_i,
\qquad \qquad 
V_j = U_j - (1-q^{-2})\sum_{i=0}^{j+1} (-1)^i u_i.
\end{align*}
\item[\rm (iv)] For $0 \leq j \leq n-1$ and $j$ odd,
\begin{align*}
&U_j = V_j + (1-q^{-2})\sum_{i=0}^{j} (-1)^i v_i,
\qquad \qquad 
V_j = U_j - (1-q^{-2})\sum_{i=0}^{j} (-1)^i u_i.
\end{align*}
\end{enumerate} 
\end{lemma}
\begin{proof} To obtain the result for case 1, use Propositions
\ref{prop:rel1},
\ref{prop:rel2},
\ref{prop:rel3} and induction on $j$.
To obtain the result for case 2 (resp. case 3) (resp. case 4),
apply $\sigma S$ (resp. $S$) (resp. $\sigma$) to everything
from case 1. 
\end{proof}
\noindent In Appendix B we present 
Lemma \ref{lem:ggww} in an alternative form.

\begin{proposition} The sets
\begin{align*}
N\cup E, \qquad
N\cup W, \qquad
S\cup E, \qquad
S\cup W
\end{align*}
all have the same span.
\end{proposition}
\begin{proof} Let $V$ denote the span of
$N\cup S \cup E\cup W$.
By case 1 or case 4 of Lemma
\ref{lem:ggww} we find that $N\cup E$ and $S\cup W$ have the
same span, which must be $V$.
By case 2 or case 3 of Lemma
\ref{lem:ggww} we find that $N\cup W$ and $S\cup E$ have the
same span, which must be $V$.
The result follows.
\end{proof}

\section{Each alternating word is contained in $U$}

\noindent In Propositions
\ref{prop:rel1},
\ref{prop:rel2},
\ref{prop:rel3}
we obtained some relations
that involve the alternating words.
In this section we obtain some additional relations for the alternating
words; these 
resemble the relations in Proposition
\ref{prop:attract}.
As we will see, these additional relations together with
Propositions
\ref{prop:rel1},
\ref{prop:rel2},
\ref{prop:rel3}
imply that   each alternating word is contained in $U$.

\begin{proposition} 
\label{prop:GGWW}
For $n\geq 1$,
\begin{align}
&
\sum_{k=0}^n  G_{k} \star \tilde G_{n-k} q^{n-2k}
= q
\sum_{k=0}^{n-1} W_{-k} \star W_{n-k} q^{n-1-2k},
\label{eq:GGWW1}
\\
&
\sum_{k=0}^n G_{k} \star \tilde G_{n-k} q^{2k-n}
= q
\sum_{k=0}^{n-1} W_{n-k} \star W_{-k} q^{n-1-2k},
\label{eq:GGWW4}
\\
&
\sum_{k=0}^n  \tilde G_{k} \star  G_{n-k} q^{n-2k}
= q
\sum_{k=0}^{n-1} W_{n-k} \star W_{-k} q^{2k+1-n},
\label{eq:GGWW2}
\\
&
\sum_{k=0}^n \tilde G_{k} \star G_{n-k} q^{2k-n}
= q
\sum_{k=0}^{n-1} W_{-k} \star W_{n-k} q^{2k+1-n}.
\label{eq:GGWW3}
\end{align}
\end{proposition}
\begin{proof} We first verify
(\ref{eq:GGWW1}) by evaluating each side.
In
(\ref{eq:GGWW1})
the $k$-summand on the left is
$\tilde G_n q^n$ (resp.
$ G_n q^{-n}$)
for $k=0$ (resp.
$k=n$).
By Lemma
 \ref{lem:free}
we have $\tilde G_n = xW_n$ and
$G_n = yW_{1-n}$.
By Lemma \ref{lem:prep},
\begin{align*}
G_{k} \star \tilde G_{n-k} &= y(W_{1-k}\star \tilde G_{n-k}) +
x(G_{k}\star W_{n-k})
\end{align*}
for $1 \leq k \leq n-1$  and
\begin{align*}
W_{-k} \star W_{n-k} &= x(G_{k}\star W_{n-k}) +
y(W_{-k}\star \tilde G_{n-k-1})q^{-2}
\end{align*}
for $0 \leq k \leq n-1$.
Using these comments, one checks that the two sides of
(\ref{eq:GGWW1}) are equal.
To obtain (\ref{eq:GGWW4}) (resp. (\ref{eq:GGWW2}))
(resp. (\ref{eq:GGWW3})),
apply $S$ (resp. $\sigma$) (resp. $\sigma S$) to everything in
(\ref{eq:GGWW1}).
\end{proof}

\begin{lemma} 
\label{lem:recgen}
Using the equations below, the alternating words in $\mathbb V$
are recursively obtained from $x, y$ in the following order:
\begin{align*}
W_0, \quad W_1, \quad G_1, \quad \tilde G_1, \quad W_{-1}, \quad W_2, \quad
 G_2, \quad \tilde G_2, \quad W_{-2}, \quad W_3, \quad \ldots
\end{align*}
We have $W_0=x$ and $W_1=y$.
For $n\geq 1$,
\begin{align}
G_n &= \frac{q\sum_{k=0}^{n-1} W_{-k}\star W_{n-k} q^{n-1-2k}
-
\sum_{k=1}^{n-1} G_k \star \tilde G_{n-k} q^{n-2k}}{q^n+q^{-n}}
+ 
\frac{W_n \star W_0-W_0\star W_n}{(1+q^{-2n})(1-q^{-2})},
\label{eq:solvG}
\\
\tilde G_n &= G_n + \frac{W_0\star W_n-W_n\star W_0}{1-q^{-2}},
\label{eq:solvGt}
\\
W_{-n} &= \frac{q W_0\star G_n - q^{-1} G_n \star W_0}{q-q^{-1}},
\label{eq:solvWm}
\\
W_{n+1} &= \frac{q G_n\star W_1 - q^{-1} W_1 \star G_n}{q-q^{-1}}.
\label{eq:solvWp}
\end{align}
\end{lemma}
\begin{proof} 
Equation
(\ref{eq:solvGt}) is from
(\ref{eq:3p1vv}).
To obtain (\ref{eq:solvG}), subtract
 $q^n$ times 
(\ref{eq:solvGt}) from 
(\ref{eq:GGWW1}), and simplify the result.
Equations
(\ref{eq:solvWm}),
(\ref{eq:solvWp}) are from
(\ref{eq:3p2vv}),
(\ref{eq:3p3vv}).
\end{proof}

\begin{theorem}
\label{prop:AltinU}
Each alternating word of $\mathbb V$ 
is contained in $U$.
\end{theorem}
\begin{proof} By Lemma
 \ref{lem:recgen}.
\end{proof}

\section{Some generating functions}

\noindent
We continue to discuss the alternating words in $\mathbb V$.
In previous sections
we found many relations involving these  words. 
In this section we express some
of these relations using generating functions.
We use these generating functions to solve for
the alternating words
(\ref{eq:G}) in terms of the alternating words
(\ref{eq:Wm}),
(\ref{eq:Wp}),
(\ref{eq:Gt}).

\begin{definition}\rm We define some generating functions
in an indeterminate $t$:
\begin{align*}
&G(t) = \sum_{n\in \mathbb N} t^n G_n,
\qquad \qquad \quad  
{\tilde G}(t) = \sum_{n\in \mathbb N} t^n {\tilde G}_n,
\\
&W^+(t) = \sum_{n\in \mathbb N} t^n W_{n+1},
\qquad \qquad 
W^-(t) = \sum_{n\in \mathbb N} t^n W_{-n}.
\end{align*}
\end{definition}

\begin{lemma}
\label{lem:gfWG}
We have
\begin{align*}
&G(q^{-1}t) \star W^-(qt) = W^-(q^{-1}t) \star G(qt),
\\
&{\tilde G}(q^{-1}t) \star W^+(qt) = W^+(q^{-1}t) \star {\tilde G}(qt),
\\
&G(qt) \star W^+(q^{-1}t) = W^+(qt) \star G(q^{-1}t).
\\
&{\tilde G}(qt) \star W^-(q^{-1}t) = W^-(qt)\star {\tilde G}(q^{-1}t),
\end{align*}
\end{lemma}
\begin{proof} Routine consequence of Proposition
\ref{prop:attract}.
\end{proof}

\begin{lemma}
\label{lem:gfGG}
We have
\begin{align*}
&G(q^{-1}t) \star {\tilde G}(qt) - qt W^- (q^{-1}t) \star W^+(qt) = 1,
\\
&G(qt) \star  {\tilde G}(q^{-1}t) - qt W^+ (qt) \star W^-(q^{-1}t) = 1,
\\
&{\tilde G}(q^{-1}t)\star  G(qt) - qt W^+ (q^{-1}t)\star W^-(qt) = 1,
\\
&{\tilde G}(qt) \star G(q^{-1}t) - qt W^- (qt)\star W^+(q^{-1}t) = 1.
\end{align*}
\end{lemma}
\begin{proof} Routine consequence of Proposition
\ref{prop:GGWW}.
\end{proof}

\begin{remark} 
\label{rem:kmat}
\rm By Lemmas
\ref{lem:gfWG},
\ref{lem:gfGG} the following matrices
are inverses with respect to the $q$-shuffle product:
\begin{align*}
        \left(
         \begin{array}{cc}
               G(q^{-1}t) & qt W^-(q^{-1}t)    
	        \\
               W^+(q^{-1}t) & \tilde G(q^{-1}t) 
                  \end{array}
              \right),
\qquad \qquad
        \left(
         \begin{array}{cc}
               \tilde G(q t) & -qt W^-(qt)    
	        \\
               -W^+(qt) &  G(qt) 
                  \end{array}
              \right).
\end{align*}
\end{remark}

\noindent Our next goal is to solve 
for $G(t)$. To this end, we introduce some elements
$\lbrace D_n \rbrace_{n\in \mathbb N}$ in $\mathbb V$.
These elements will be defined recursively.

\begin{definition}
\label{def:VG}
\rm
Define $\lbrace D_n \rbrace_{n\in \mathbb N}$ in $\mathbb V$
such that
 $D_0 = 1$ and for $n\geq 1$,
\begin{align}
\label{eq:VGzero}
D_0 \star \tilde G_n+ D_1 \star \tilde G_{n-1} + \cdots + 
D_n \star \tilde G_0 = 0.
\end{align}
\end{definition}
\begin{example}
\label{ex:GvsV}
\rm We have
\begin{align*}
D_1 &= -\tilde G_1,
\\
D_2 &= \tilde G_1\star \tilde G_1 -\tilde G_2,
\\
D_3 &= 2 \tilde G_1\star \tilde G_2 - \tilde G_1 \star \tilde G_1 \star 
\tilde G_1 - \tilde G_3,
\\
D_4 &= 
\tilde G_1 \star 
\tilde G_1 \star 
\tilde G_1 \star 
\tilde G_1 
+ 2\tilde G_1 \star  \tilde G_3 + \tilde G_2 \star \tilde G_2-
3\tilde G_1 \star \tilde G_1\star 
\tilde G_2 -\tilde G_4
\end{align*}
and
\begin{align*}
\tilde G_1 &= -D_1,
\\
\tilde G_2 &= D_1 \star D_1 -D_2,
\\
\tilde G_3 &= 2 D_1 \star  D_2 - D_1 \star D_1 \star  D_1 - D_3,
\\
\tilde G_4 &= 
D_1 \star
D_1 \star
D_1 \star
D_1
+ 2 D_1 \star D_3 + D_2 \star D_2-3 D_1 \star D_1 \star 
D_2 -D_4.
\end{align*}
\end{example}

\begin{lemma}
\label{prop:VCpoly}
For $n\geq 1$ the following hold in the $q$-shuffle algebra $\mathbb V$.
\begin{enumerate}
\item[\rm (i)] $D_n$ is a homogeneous polynomial in 
$\tilde G_1, \tilde G_2,\ldots, \tilde G_n$ that has total degree $n$,
where we view each $\tilde G_i$ as having degree $i$.
In this polynomial the coefficient of
$\tilde G_n$ is $-1$.
\item[\rm (ii)] $\tilde G_n$ is a homogeneous polynomial in 
$D_1, D_2,\ldots, D_n$ that has total degree $n$, where we view
each $D_i$ as having degree $i$.
In this polynomial the coefficient of
$D_n$ is $-1$.
\end{enumerate}
\end{lemma}
\begin{proof} By Definition 
\ref{def:VG} and induction on $n$.
\end{proof}

\begin{lemma}
\label{lem:samesub}
The following coincide:
\begin{enumerate}
\item[\rm (i)] the subalgebra of the $q$-shuffle algebra $\mathbb V$
generated by $\lbrace D_n \rbrace_{n =1}^\infty$;
\item[\rm (ii)] the subalgebra of the $q$-shuffle algebra $\mathbb V$
generated by $\lbrace \tilde  G_n \rbrace_{n =1}^\infty$.
\end{enumerate}
\end{lemma}
\begin{proof} Each subalgebra 
contains the other by 
Lemma \ref{prop:VCpoly}.
\end{proof}

\noindent We emphasize a few points.

\begin{lemma} For $n \in \mathbb N$ we have $D_n \in U$.
\end{lemma}
\begin{proof} By 
Theorem
\ref{prop:AltinU}
and
Lemma
\ref{lem:samesub}.
\end{proof}

\begin{lemma} 
\label{lem:VGprop1}
For $i,j\in \mathbb N$ the following holds in
the $q$-shuffle algebra $\mathbb V$:
\begin{align*}
\lbrack D_i, D_j\rbrack = 0, \qquad \qquad
\lbrack D_i, \tilde G_j\rbrack = 0.
\end{align*}
\end{lemma}
\begin{proof} By 
Lemma \ref{lem:samesub}
and since the $\lbrace \tilde G_j\rbrace_{j\in \mathbb N}$
mutually commute.
\end{proof}

\begin{definition}\rm We define a generating function
in the indeterminate $t$:
\begin{align*}
D(t) = \sum_{n\in \mathbb N} t^n D_n.
\end{align*}
\end{definition}

\begin{lemma} 
\label{lem:VGinv}
We have
\begin{align*}
D(t) \star \tilde G(t) = 1 = \tilde G(t) \star  D(t).
\end{align*}
\end{lemma} 
\begin{proof} Use Definition
\ref{def:VG} and Lemma
\ref{lem:VGprop1}.
\end{proof}

\begin{lemma}
\label{lem:VW}
We have
\begin{align*}
&
W^-(q^{-1}t)\star D(q^{-1}t) = 
D(qt)\star W^-(qt), 
\\
&
D(q^{-1}t)\star W^+(q^{-1}t) = 
W^+(qt)\star D(qt).
\end{align*}
\end{lemma}
\begin{proof} 
By the two equations in Lemma
\ref{lem:gfWG} that involve $\tilde G$, together with
Lemma
\ref{lem:VGinv}.
\end{proof}

\begin{proposition}
\label{prop:solveG}
We have
\begin{align}
G(t) &= D(q^2t) + q^2t W^-(t) \star D(t) \star W^+(t),
\label{eq:GV1}
\\
G(t) &= D(q^{-2}t) + t W^+(t) \star D(t) \star W^-(t).
\label{eq:GV2}
\end{align}
\end{proposition}
\begin{proof} 
In the equations of Lemma
\ref{lem:gfGG}, eliminate the $\tilde G$ term using
Lemma
\ref{lem:VGinv}, 
evaluate the result using 
Lemma \ref{lem:VW},
and then
make a change of variables $t\mapsto q^{\pm 1}t$ as needed.
\end{proof}

\begin{theorem}
\label{prop:Gexpl} 
For $n\in \mathbb N$ we have
\begin{align}
\label{eq:Gn1}
G_n 
&= q^{2n}D_n +
q^2 \sum_{\stackrel{i+j+k+1=n}{i,j,k\geq 0}}
W_{-i} \star D_{j} \star W_{k+1},
\\
G_n &= q^{-2n}D_n +
 \sum_{\stackrel{i+j+k+1=n}{i,j,k\geq 0}}
W_{k+1} \star D_{j} \star W_{-i}.
\label{eq:Gn2}
\end{align}
\end{theorem}
\begin{proof} For the equations
(\ref{eq:GV1}) and
(\ref{eq:GV2}),
compare the
coeffient of $t^n$ on either side.
\end{proof}

\begin{example}
\label{ex:Gelim}
\rm
For $1 \leq n \leq 4$ the equations
(\ref{eq:Gn1}), 
(\ref{eq:Gn2}) are given below:
\begin{align*}
 G_1 &= q^2D_1 + 
 q^2 W_0\star W_1,
\\
G_2 &= q^4D_2+
q^2 W_0 \star W_2 
+
q^2 W_{-1} \star W_1 
+
q^2 W_0 \star D_1 \star  W_1,
\\
 G_3 &= q^6D_3+
q^2 W_0 \star W_3+
q^2 W_{-1}\star W_2+
q^2 W_{-2}\star W_1+
q^2 W_0 \star D_1 \star W_2
\\
& \quad + q^2 W_{-1} \star D_1 \star W_1
+ q^2 W_0 \star D_2 \star W_1,
\\
G_4 &= q^8 D_4+
q^2 W_{0} \star W_4 +
q^2 W_{-1} \star W_3 +
q^2 W_{-2} \star W_2 +
q^2 W_{-3} \star W_1 
\\& \quad+
q^2 W_{0} \star D_1 \star W_3 +
q^2 W_{-1} \star D_1 \star W_2 +
q^2 W_{-2} \star D_1 \star W_1 +
q^2 W_{0} \star D_2 \star W_2
\\& \quad+
q^2 W_{-1} \star D_2 \star W_1+
 q^2 W_{0} \star D_3 \star W_1,
\end{align*}
\noindent and
\begin{align*}
G_1 &= q^{-2}D_1 + 
  W_1\star W_0,
\\
G_2 &= q^{-4}D_2+
 W_1 \star W_{-1} 
+
 W_{2} \star W_0 
+
 W_1 \star D_1 \star  W_0,
\\
 G_3 &= q^{-6} D_3+
 W_1 \star W_{-2}+
 W_2\star W_{-1}+
 W_3\star W_0+
 W_1 \star D_1 \star W{-1}
\\
& \quad +  W_{2} \star D_1 \star W_0
+  W_1 \star D_2 \star W_0,
\\
G_4 &= q^{-8} D_4+
 W_{1} \star W_{-3} +
 W_{2} \star W_{-2} +
 W_{3} \star W_{-1}+
 W_{4} \star W_0 
\\& \quad+
 W_{1} \star D_1 \star W_{-2} +
 W_{2} \star D_1 \star W_{-1} +
 W_{3} \star D_1 \star W_{0} +
 W_{1} \star D_2 \star W_{-1}
\\
& \quad+
 W_{2} \star D_2 \star W_0+
 W_{1} \star D_3 \star W_0.
\end{align*}
\noindent Note that $D_1,D_2,D_3,D_4$
are given in Example
\ref{ex:GvsV}.
\end{example}

\section{The alternating PBW basis for $U$}

\noindent In this section we obtain some PBW bases for
$U$, including the alternating PBW basis.

\begin{theorem}
\label{thm:secondkind}
A PBW basis for $U$ is obtained by the elements
\begin{align}
\label{eq:omega}
\lbrace W_{-i} \rbrace_{i \in \mathbb N}, \qquad 
\lbrace \tilde G_{j+1} \rbrace_{j\in \mathbb N}, \qquad  
\lbrace W_{k+1} \rbrace_{k\in \mathbb N}
\end{align}
in any linear order $<$ that satisfies one of {\rm (i)--(vi)} below:
\begin{enumerate}
\item[\rm (i)]  $W_{-i} <  \tilde G_{j+1} < W_{k+1}$ for $i,j,k\in \mathbb N$;
\item[\rm (ii)]  $W_{k+1} < \tilde G_{j+1} < W_{-i}$ for $i,j,k\in \mathbb N$;
\item[\rm (iii)]  $W_{k+1} < W_{-i} < \tilde G_{j+1}$ for $i,j,k\in \mathbb N$;
\item[\rm (iv)]  $W_{-i} < W_{k+1} < \tilde G_{j+1}$ for $i,j,k\in \mathbb N$;
\item[\rm (v)]  $\tilde G_{j+1} < W_{k+1} < W_{-i}$ for $i,j,k\in \mathbb N$;
\item[\rm (vi)]  $\tilde G_{j+1} < W_{-i} < W_{k+1}$ for $i,j,k\in \mathbb N$.
\end{enumerate}
\end{theorem}
\begin{proof} (i)
We invoke Proposition
\ref{prop:useful}.
For notational convenience we identify
the algebras $U^+_q$ and $U$, via the isomorphism $\natural$
from below 
(\ref{eq:qsc2}).
Let $\Omega$ 
 denote the set of alternating words listed in
 (\ref{eq:omega}).
We have $\Omega \subseteq U$ by Theorem
\ref{prop:AltinU}.
We show that $\Omega$ is feasible in the sense of
Definition
\ref{def:omega}. 
Definition
\ref{def:omega}(i) holds since
each element of $\Omega$ is a word in $\mathbb V$.
Definition \ref{def:omega}(ii) holds 
by the table above Lemma
\ref{lem:sigSact}.
We have shown that $\Omega$ is feasible.
The set $\Omega$ has a linear order $<$ from the theorem statement.
We show that $\Omega$ in order $<$ satisfies
Proposition \ref{prop:useful}(ii). 
In the present notation, we must show that the
vector space $U$ is spanned by
\begin{align}
a_1 \star a_2 \star \cdots \star a_n \qquad \quad 
n \in \mathbb N, \qquad \quad
a_1, a_2, \ldots, a_n \in \Omega, \qquad \quad
a_1 \leq a_2 \leq \cdots \leq a_n.
\label{eq:Ubasis}
\end{align}
Let $\mathbb A$ denote the set of alternating words in $\mathbb V$.
So $\mathbb A$ 
is the union of $\Omega$ and
$\lbrace G_{\ell+1}\rbrace_{\ell\in \mathbb N}$.
The generators $x,y$ of $U$ are contained in $\mathbb A$,
so the algebra $U$ is generated by 
$\mathbb A$.
Therefore the vector space $U$ is spanned by
\begin{align}
a_1\star a_2 \star \cdots \star a_n \qquad \quad n\in \mathbb N,
\qquad \quad a_1, a_2, \ldots, a_n \in \mathbb A.
\label{eq:start}
\end{align}
We have a linear order $<$ on $\Omega$; extend $<$
to $\mathbb A$ such that
\begin{align}
W_{-i} < G_{\ell+1} < \tilde G_{j+1} < W_{k+1} \qquad \qquad i,j,k,\ell
\in \mathbb N.
\label{eq:extend}
\end{align}
For every product 
$a_1 \star  a_2 \star \cdots \star a_n$ in
(\ref{eq:start}) and for every pair $a_{s-1},a_s$ $(2 \leq s\leq n)$
of adjacent terms such that $a_{s-1} > a_s$,
we can use 
the commutator relations from 
Sections
6, 7 to express $a_{s-1}\star a_s$ as a linear combination of
products $a'_{s-1}\star a'_s$ such that $a'_{s-1}, a'_s \in \mathbb A$
and $a'_{s-1} \leq a'_s$.
 Specifically,
use 
Lemma
\ref{lem:transGW} (resp. 
 Lemma
\ref{lem:ggww}) 
if the lengths of $a_{s-1}$ and
$a_s$ have opposite (resp. same)
parity.   
This argument shows that 
the vector space $U$ is spanned by the products
\begin{align*}
a_1 \star a_2 \star \cdots \star a_n \qquad \quad n\in \mathbb N,
\qquad \quad a_1, a_2, \ldots, a_n \in \mathbb A,\qquad \quad
a_1 \leq a_2 \leq \cdots \leq a_n.
\end{align*}
Let 
 $a_1 \star a_2 \star \cdots \star a_n$ denote one of the above products.
We show that
 $a_1\star a_2\star \cdots \star a_n$ is contained in 
the span of
(\ref{eq:Ubasis}).
Our proof is by induction on
the number $\zeta$ of terms among
$a_1,a_2,\ldots,a_n$ that are contained in
$\lbrace G_{\ell+1}\rbrace_{\ell \in \mathbb N}$.
First assume that
$\zeta=0$.  Then $a_i \in \Omega$  for $1 \leq i \leq n$, so
$a_1\star a_2 \star \cdots \star a_n$ is listed in
(\ref{eq:Ubasis}).
Next assume that $\zeta \geq 1$. 
At least one of $a_1,a_2,\ldots,a_n$ is contained in 
$\lbrace G_{\ell+1}\rbrace_{\ell \in \mathbb N}$.
Pick the maximal integer $s$ such that $1 \leq s\leq n$ 
and $a_s$ is contained in
$\lbrace G_{\ell+1}\rbrace_{\ell \in \mathbb N}$.
Eliminate $a_s$ using 
(\ref{eq:Gn1}) and straighten the result using
the commutator  relations in
Lemma
\ref{lem:transGW}.
These moves and induction show that 
$a_1\star a_2\star \cdots \star a_n$ 
is a linear combination of vectors, 
each contained
in the span of
(\ref{eq:Ubasis}).
So  $a_1\star a_2\star \cdots \star a_n$
is contained in the span of
(\ref{eq:Ubasis}), as desired.
By the above comments the vector space $U$ is spanned
by 
(\ref{eq:Ubasis}).
 Proposition 
\ref{prop:useful}(ii) is now satisfied. By that proposition
we find that $\Omega$ in order $<$ is a PBW basis for $U$.
\\
\noindent (ii) Similar to the proof of (i) above, with
line (\ref{eq:extend}) replaced by
\begin{align*}
W_{k+1} <  G_{\ell+1} < \tilde G_{j+1} < W_{-i} \qquad \qquad i,j,k,\ell
\in \mathbb N,
\end{align*}
and use 
(\ref{eq:Gn2})
instead of
(\ref{eq:Gn1}).
\\
\noindent (iii)  By (ii) above and case 3
of
Proposition
\ref{prop:table1}.
\\
\noindent (iv)  By (i) above and case 4 of Proposition
\ref{prop:table1}.
\\
\noindent (v)  By (ii) above and case 4
of Proposition
\ref{prop:table1}.
\\
\noindent (vi) By (i) above and case 3 of Proposition
\ref{prop:table1}.
\end{proof}

\begin{theorem}
\label{thm:firstkind}
A PBW basis for $U$ is obtained by the elements
\begin{align*}
\lbrace W_{-i} \rbrace_{i \in \mathbb N}, \qquad 
\lbrace  G_{j+1} \rbrace_{j\in \mathbb N}, \qquad  
\lbrace W_{k+1} \rbrace_{k\in \mathbb N}
\end{align*}
in any linear order $<$ that satisfies one of {\rm (i)--(vi)} below:
\begin{enumerate}
\item[\rm (i)]  $W_{-i} <  G_{j+1} < W_{k+1}$ for $i,j,k\in \mathbb N$;
\item[\rm (ii)]  $W_{k+1} < G_{j+1} < W_{-i}$ for $i,j,k\in \mathbb N$;
\item[\rm (iii)]  $W_{k+1} < W_{-i} < G_{j+1}$ for $i,j,k\in \mathbb N$;
\item[\rm (iv)]  $W_{-i} < W_{k+1} < G_{j+1}$ for $i,j,k\in \mathbb N$;
\item[\rm (v)]  $G_{j+1} < W_{k+1} < W_{-i}$ for $i,j,k\in \mathbb N$;
\item[\rm (vi)]  $G_{j+1} < W_{-i} < W_{k+1}$ for $i,j,k\in \mathbb N$.
\end{enumerate}
\end{theorem}
\begin{proof} Apply the automorphism $\sigma$ to everything in
Theorem
\ref{thm:firstkind}, and use Lemma
\ref{lem:sigSact}(i).
\end{proof}

\noindent In our view, the above twelve PBW bases for $U$ are not
substantially different. So we focus on
the most convenient one, which is from
Theorem \ref{thm:secondkind}(i).

\begin{definition}\rm The 
{\it alternating PBW basis for $U$} consists of the elements
\begin{align*}
\lbrace W_{-i} \rbrace_{i \in \mathbb N}, \qquad 
\lbrace \tilde G_{j+1} \rbrace_{j\in \mathbb N}, \qquad  
\lbrace W_{k+1} \rbrace_{k\in \mathbb N}
\end{align*}
in a linear order $<$ that satisfies 
\begin{align*}
W_{-i} <  \tilde G_{j+1} < W_{k+1}  \qquad i,j,k\in \mathbb N.
\end{align*}
\end{definition}

\section{Comparing the
Damiani PBW basis and the
alternating PBW basis }

\noindent 
In Section 2 we described the Damiani PBW basis. In the
previous section we defined the alternating PBW basis.
In this section we show how these two PBW bases are related.
\medskip

\noindent We will adopt the following point of view.
Instead of working directly with the Damiani PBW basis
   (\ref{eqUq:PBWintro}), we will work with 
   the closely related elements
   $\lbrace xC_n\rbrace_{n=0}^\infty$,
   $\lbrace C_ny\rbrace_{n=0}^\infty$,
   $\lbrace C_n\rbrace_{n=1}^\infty$
   from the end of Section 4.  We begin with 
   $\lbrace xC_n\rbrace_{n=0}^\infty$ and
   $\lbrace C_ny\rbrace_{n=0}^\infty$.

\begin{definition}\rm We define some generating functions
in the indeterminate $t$:
\begin{align*}
C^-(t) = \sum_{n \in \mathbb N} t^n (xC_n),
\qquad \qquad 
C^+(t) = \sum_{n \in \mathbb N} t^n (C_n y).
\end{align*}
\end{definition}

\begin{proposition}
\label{prop:C4}
We have
\begin{align*}
 &C^-(-t) =
 W^-(q^{-1}t)\star D(q^{-1}t) = 
 D(qt)\star W^-(qt),
\\
 &C^+(-t) = D(q^{-1}t)\star W^+(q^{-1}t)=
W^+(qt) \star D(qt).
\end{align*}
\end{proposition}
\begin{proof} We first verify 
 $C^-(-t) =
 W^-(q^{-1}t)\star D(q^{-1}t)$.
To do this, it suffices to show that for $n\in \mathbb N$,
\begin{align}
x C_n = (-1)^n q^{-n} \sum_{i=0}^n W_{-i} \star D_{n-i}.
\label{eq:xCn}
\end{align} Let 
$\widehat{xC}_n$ denote the expression on the right in
(\ref{eq:xCn}). 
We show that $xC_n = \widehat{xC}_n$.
We will use induction on $n$. The result holds for $n=0$,
since 
$xC_0 =x= \widehat{xC}_0$.
Next assume that $n\geq 1$.
By the equation on the left in 
 (\ref{eq:dam1introalt}) together with the discussion at the
 end of Section 4, 
\begin{align}
(q-q^{-1})(xC_n) = (x C_{n-1}) \star (xy) - (xy) \star (x C_{n-1}).
\label{eq:xCR}
\end{align}
\noindent By induction,
the right-hand side of 
(\ref{eq:xCR}) is equal to
\begin{align*}
\widehat{x C}_{n-1} \star (xy) - (xy) \star \widehat{x C}_{n-1},
\end{align*}
which by construction and $xy=\tilde G_1$ is equal to
\begin{align*}
(-1)^{n-1} q^{1-n} \sum_{i=0}^{n-1}( W_{-i} \star D_{n-i-1}
\star \tilde G_1 - \tilde G_1 \star W_{-i} \star D_{n-i-1}),
\end{align*}
which by Lemma 
\ref{lem:VGprop1}
is equal to
\begin{align*}
(-1)^{n-1} q^{1-n} \sum_{i=0}^{n-1}( W_{-i} \star \tilde G_1
 - \tilde G_1 \star W_{-i}) \star D_{n-i-1},
\end{align*}
which by 
(\ref{eq:3p7vv}) is equal to
\begin{align*}
(-1)^{n-1} q^{1-n} \sum_{i=0}^{n-1}( W_{0} \star \tilde G_{i+1}
 - \tilde G_{i+1} \star W_{0}) \star D_{n-i-1},
\end{align*}
which
after a change of variables $i\mapsto i-1$ is equal to
\begin{align*}
(-1)^{n-1} q^{1-n} \sum_{i=0}^{n}( W_{0} \star \tilde G_{i}
 - \tilde G_{i} \star W_{0}) \star D_{n-i},
\end{align*}
which by $\tilde G_0=1$ and 
(\ref{eq:3p2vv}) 
is equal to 
\begin{align*}
(-1)^{n-1} q^{1-n} (1-q^{-2})
\sum_{i=0}^{n}( W_{0} \star \tilde G_{i} -W_{-i} ) \star D_{n-i},
\end{align*}
which by 
(\ref{eq:VGzero}) and algebra  
is equal to 
\begin{align*}
(-1)^{n} q^{-n} (q-q^{-1})
\sum_{i=0}^{n}W_{-i}  \star D_{n-i},
\end{align*}
which is equal to $(q-q^{-1})\widehat{xC}_n$.
We have shown that
$xC_n = \widehat{xC}_n$, so
(\ref{eq:xCn}) holds. We have verified
 $C^-(-t) =
 W^-(q^{-1}t)\star D(q^{-1}t)$.
In this equation apply $\sigma S$ to each side, to get
 $C^+(-t) =
 D(q^{-1}t)\star W^+(q^{-1}t)$.
The remaining equations in the proposition
statement are from Lemma
\ref{lem:VW}.
\end{proof}

\noindent Next we restate
Proposition
\ref{prop:C4} using the style of
(\ref{eq:xCn}).

\begin{theorem} For $n \in \mathbb N$,
\begin{align*}
xC_n &=
(-1)^n q^{-n} \sum_{i=0}^n W_{-i}\star D_{n-i}
=
(-1)^n q^n \sum_{i=0}^n D_{n-i} \star W_{-i},
\\
C_ny &= 
(-1)^n q^{-n} \sum_{i=0}^n D_{n-i} \star W_{i+1}
=
(-1)^n q^n \sum_{i=0}^n W_{i+1} \star D_{n-i}.
\end{align*}
\end{theorem}
\begin{proof} For each equation in Proposition
\ref{prop:C4}, compare the coefficient of $t^n$
on either side.
\end{proof}

\begin{proposition} 
\label{prop:W4}
We have
\begin{align*}
   &W^-(t) = C^-(-qt) \star \tilde G(t) = 
   \tilde G(t) \star C^-(-q^{-1}t),
   \\
   &W^+(t) = \tilde G(t)\star C^+(-q t) =
     C^+(-q^{-1}t)\star \tilde G(t).
\end{align*}
\end{proposition}
\begin{proof} Evaluate the equations in Proposition
\ref{prop:C4}
using 
Lemma \ref{lem:VGinv}.
\end{proof}

\begin{theorem}
\label{prop:WCG}
For $n \in \mathbb N$,
\begin{align*}
W_{-n}
= \sum_{i=0}^n (-1)^i q^i (xC_i) \star \tilde G_{n-i}
= \sum_{i=0}^n (-1)^i q^{-i} \tilde G_{n-i} \star (xC_i),
\\
W_{n+1}
= \sum_{i=0}^n (-1)^i q^i \tilde G_{n-i} \star (C_i y)
= \sum_{i=0}^n (-1)^i q^{-i} (C_i y) \star \tilde G_{n-i}.
\end{align*}
\end{theorem}
\begin{proof} For each equation in
Proposition
\ref{prop:W4}, compare the coefficients of $t^n$ on either side.
\end{proof}

\noindent Now we bring in $\lbrace C_n\rbrace_{n=1}^\infty$.
Before getting into detail
we emphasize one point.

\begin{lemma}
\label{lem:Ccom}
In the $q$-shuffle algebra $\mathbb V$
the elements
$\lbrace C_n\rbrace_{n\in \mathbb N}$ mutually commute.
\end{lemma}
\begin{proof} By the comment below Proposition
    \ref{prop:PBWbasis} and the discussion at the end of Section 4.
    \end{proof}

\begin{definition}\rm We define a generating function
in the indeterminate $t$:
\begin{align*}
C(t) &= \sum_{n \in \mathbb N} t^n C_n.
\end{align*}
\end{definition}

\begin{proposition} 
\label{prop:CVV}
We have
\begin{align}
&C(-t) = D(qt) \star D(q^{-1}t).
\label{eq:CVsV}
\end{align}
\end{proposition}
\begin{proof} 
We first show that for $n \in \mathbb N$,
\begin{align}
G_n = \sum_{i=0}^n (-1)^i q^i C_i \star \tilde G_{n-i} +
q^2 \sum_{i=0}^{n-1} (-1)^i q^i (xC_i) \star W_{n-i}.
\label{eq:CVV1}
\end{align}
For $n=0$ the equation
(\ref{eq:CVV1}) holds, since each side is equal to 1.
Next assume  that $n\geq 1$.
By the first equation in Theorem
\ref{prop:WCG},
\begin{align}
W_{-n} =
(-1)^n q^n xC_n+ 
 \sum_{i=0}^{n-1} (-1)^i q^i (xC_i)\star \tilde G_{n-i}.
\label{eq:CVV2}
\end{align}
\noindent We will evaluate 
(\ref{eq:CVV2})  after some comments.
By Lemma
\ref{lem:free}  we have $W_{-n} = xG_n$.
Also by 
 Lemma
\ref{lem:free}
we have $\tilde G_{n-i}= xW_{n-i} $ for $0 \leq i \leq n-1$.
By this and (\ref{eq:uvcirc}),
\begin{align}
(xC_i) \star \tilde G_{n-i} = x\bigl(
C_i \star \tilde G_{n-i} + q^2 (xC_i)\star W_{n-i} \bigr)
\qquad \qquad (0 \leq i \leq n-1).
\label{eq:idea1}
\end{align}
\noindent Evaluating 
(\ref{eq:CVV2}) using the above comments, we obtain
\begin{align*}
xG_n = x \sum_{i=0}^n (-1)^i q^i C_i \star \tilde G_{n-i} + 
q^2 x \sum_{i=0}^{n-1} (-1)^i q^i (xC_i) \star W_{n-i}.
\end{align*}
In the above equation, each term has an $x$
on the left;
removing such $x$ we obtain 
(\ref{eq:CVV1}). 
In terms of generating functions,
(\ref{eq:CVV1}) becomes
\begin{align*}
G(t) = C(-qt) \star \tilde G(t) + q^2t C^-(-qt) \star W^+(t).
\end{align*}
Adjusting this equation using the first equation in Proposition
\ref{prop:C4}, we obtain
\begin{align}
G(t) = C(-qt) \star \tilde G(t) + q^2t W^-(t) \star D(t) \star W^+(t).
\label{eq:CVV3}
\end{align}
Comparing 
(\ref{eq:GV1}),
(\ref{eq:CVV3})
we obtain
$D(q^2t)= C(-qt) \star \tilde G(t)$. In this equation 
replace $t$ by $q^{-1}t$ and use
Lemma
\ref{lem:VGinv} to obtain
(\ref{eq:CVsV}).
\end{proof}

\begin{proposition} 
\label{prop:CVVform}
For $n \in \mathbb N$,
\begin{align*}
C_n = (-1)^n \sum_{i=0}^n q^{2i-n} D_i \star D_{n-i}
\end{align*}
\end{proposition}
\begin{proof} 
In the equation (\ref{eq:CVsV}) compare the coefficient
of $t^n$ on either side.
\end{proof}

\begin{corollary}
\label{prop:CVpoly}
For $n\geq 1$ the following hold in the $q$-shuffle algebra $\mathbb V$.
\begin{enumerate}
\item[\rm (i)] $C_n$ is a homogeneous polynomial in 
$D_1, D_2, \ldots, D_n$ that has total degree $n$,
where we view each $D_i$ as having degree $i$.
In this polynomial the coefficient of
$D_n$ is $(-1)^n (q^n+q^{-n})$.
\item[\rm (ii)] $D_n$ is a homogeneous polynomial in 
$C_1, C_2,\ldots, C_n$ that has total degree $n$, where we view
each $C_i$ as having degree $i$.
In this polynomial the coefficient of
$C_n$ is $(-1)^n(q^n+q^{-n})^{-1}$.
\end{enumerate}
\end{corollary}
\begin{proof} (i)  By Proposition
\ref{prop:CVVform}.
\\
\noindent (ii) By (i) above and induction on $n$.
\end{proof}

\begin{corollary}
\label{prop:CGpoly}
For $n\geq 1$ the following hold in the $q$-shuffle algebra $\mathbb V$.
\begin{enumerate}
\item[\rm (i)] $C_n$ is a homogeneous polynomial in 
$\tilde G_1, \tilde G_2,\ldots, \tilde G_n$ that has total degree $n$,
where we view each $\tilde G_i$ as having degree $i$.
In this polynomial the coefficient of
$\tilde G_n$ is $(-1)^{n+1} (q^n+q^{-n})$.
\item[\rm (ii)] $\tilde G_n$ is a homogeneous polynomial in 
$C_1, C_2,\ldots, C_n$ that has total degree $n$, where we view
each $C_i$ as having degree $i$.
In this polynomial the coefficient of
$C_n$ is $(-1)^{n+1}(q^n+q^{-n})^{-1}$.
\end{enumerate}
\end{corollary}
\begin{proof} Combine
Lemma \ref{prop:VCpoly} and
and Corollary \ref{prop:CVpoly}.
\end{proof}

\begin{corollary} 
\label{cor:CVGt}
The following {\rm (i)--(iii)} coincide:
\begin{enumerate}
\item[\rm (i)] the subalgebra of the $q$-shuffle algebra $\mathbb V$
generated by $\lbrace C_n \rbrace_{n=1}^\infty$;
\item[\rm (ii)] the subalgebra of the $q$-shuffle algebra $\mathbb V$
generated by $\lbrace D_n \rbrace_{n=1}^\infty$;
\item[\rm (iii)] the subalgebra of the $q$-shuffle algebra $\mathbb V$
generated by $\lbrace \tilde G_n \rbrace_{n=1}^\infty$.
\end{enumerate}
\end{corollary}
\begin{proof} 
By Lemma
\ref{lem:samesub}
and
Corollary
\ref{prop:CGpoly}.
\end{proof}

\noindent 
In Corollary
\ref{cor:CVGt} we saw that 
$\lbrace C_n\rbrace_{n=1}^\infty $
and 
$\lbrace \tilde G_n\rbrace_{n=1}^\infty $
generate the same subalgebra of the $q$-shuffle algebra $\mathbb V$.
Next we discuss
in more detail how
$\lbrace C_n\rbrace_{n=1}^\infty $
and 
$\lbrace \tilde G_n\rbrace_{n=1}^\infty $
are related.

\begin{lemma}
\label{prop:CVC}
\noindent We have
\begin{align*}
 C(-qt)\star \tilde G(q^2t) = D(t) = 
C(-q^{-1}t) \star \tilde G(q^{-2}t).
\end{align*}
\end{lemma}
\begin{proof} To verify the above equations, eliminate
$C(-qt)$ and 
$C(-q^{-1}t)$ using
Proposition
\ref{prop:CVV}, and evaluate the result using
Lemma
\ref{lem:VGinv}.
\end{proof}

\begin{theorem} 
\label{prop:CG2}
For $n \in \mathbb N$,
\begin{align}
0 = \sum_{i=0}^n (-1)^i \lbrack 2n-i\rbrack_q C_i  \star \tilde G_{n-i}.
\label{eq:finalRec}
\end{align}
\end{theorem}
\begin{proof} By Lemma
\ref{prop:CVC},
\begin{align*}
 C(-qt)\star \tilde G(q^2t) 
= C(-q^{-1}t)\star \tilde G(q^{-2}t).
\end{align*}
In this equation, compare the coefficient of $t^n$ on either side.
\end{proof}

\begin{remark}\rm  Using 
(\ref{eq:finalRec}) we can recursively solve for 
$\lbrace C_n\rbrace_{n=1}^\infty $
in terms of $\lbrace \tilde G_n\rbrace_{n=1}^\infty $, 
 and also
$\lbrace \tilde G_n\rbrace_{n=1}^\infty $ in terms of
$\lbrace C_n\rbrace_{n=1}^\infty$.
\end{remark}

\section{Some comments about Propositions
\ref{prop:attract} and
\ref{prop:GGWW}
}

Consider the equations in Propositions
\ref{prop:attract}  and
\ref{prop:GGWW}. For each equation,
it is natural
to ask what is the common value of each side, in terms
of the standard basis for $\mathbb V$.  In this section
we compute this common value, and give some additional
results of a similar nature.
\medskip

\noindent Near the end of Section 4 we defined the Catalan words
in $\mathbb V$, using the notation $\overline x=1$
and $\overline y=-1$. We now use this notation to define another
kind of word.

\begin{definition}
\label{def:const} 
\rm  A word $v_1v_2\cdots v_n$ in $\mathbb V$ is {\it constrained}
whenever 
$\overline v_1+
\overline v_2+\cdots  + 
\overline v_i \in \lbrace 0,\pm 1\rbrace $ 
for $1 \leq i \leq n-1$ and 
$\overline v_1+
\overline v_2+\cdots  + 
\overline v_n =0$. 
In this case $n$ is even.
\end{definition}

\begin{example}
\label{ex:constrEx}
\rm For $0 \leq n \leq 3$ we give the constrained words
of length $2n$.
\bigskip

\centerline{
\begin{tabular}[t]{c|c}
   $n$  & {\rm constained words of length $2n$}
   \\
   \hline
 $ 0 $  &  $1$
 \\
 $ 1 $  &  $xy, \quad yx$  
 \\
 $ 2 $  &  $xyxy, \quad xyyx, \quad yxxy, \quad yxyx$
 \\
 $ 3 $  &
 $
 xyxyxy, \quad
 xyxyyx, \quad
 xyyxxy, \quad
 xyyxyx, \quad 
 yxxyxy, \quad 
 yxxyyx, \quad 
 yxyxxy, \quad 
 yxyxyx
 $
   \end{tabular}}
\end{example}

\begin{lemma} For $n \in \mathbb N$ there are
$2^n$ constrained words of
length $2n$. These words have the form $b_1b_2\cdots b_n$ with
$b_i \in \lbrace xy, yx\rbrace$ for $1 \leq i \leq n$.
\end{lemma}

\noindent For $n \in \mathbb N$ consider the sum of the
constrained words that have length $2n$.

\begin{lemma} For $n \in \mathbb N$ the above sum is equal to
$(xy+yx)^n$, where the exponent is with respect to the
free product.
\end{lemma}

\noindent The following results can be obtained using
Lemma
\ref{lem:prep}
and induction on $n$. The proofs are straightforward and omitted.

\begin{proposition} 
\label{prop:attractEx}
For $n \in \mathbb N$,
\begin{align*}
&
\sum_{k=0}^n  G_{n-k} \star W_{-k} q^{2k-n} = 
\sum_{k=0}^n W_{-k} \star  G_{n-k} q^{n-2k} =
\lbrack 2 \rbrack^n_q (xy+yx)^n x,
\\
&
\sum_{k=0}^n  G_{n-k} \star W_{k+1} q^{n-2k} = 
\sum_{k=0}^n W_{k+1} \star  G_{n-k} q^{2k-n} = 
\lbrack 2 \rbrack^n_q y(xy+yx)^n,
\\
&
\sum_{k=0}^n \tilde G_{n-k} \star W_{-k} q^{n-2k} = 
\sum_{k=0}^n W_{-k} \star \tilde G_{n-k} q^{2k-n} = 
\lbrack 2 \rbrack^n_q  x (xy+yx)^n,
\\
&
\sum_{k=0}^n  \tilde G_{n-k} \star W_{k+1} q^{2k-n} = 
\sum_{k=0}^n W_{k+1} \star  \tilde G_{n-k} q^{n-2k} = 
\lbrack 2 \rbrack^n_q (xy+yx)^n y.
\end{align*}
\end{proposition}

\begin{proposition} 
\label{prop:GGWWE}
For $n\geq 1$,
\begin{align*}
&
\sum_{k=0}^n  G_{k} \star \tilde G_{n-k} q^{n-2k}
= q
\sum_{k=0}^{n-1} W_{-k} \star W_{n-k} q^{n-1-2k} 
= \lbrack 2 \rbrack^{n-1}_q (xy+yx)^{n-1} (qxy+q^{-1}yx), 
\\
&
\sum_{k=0}^n G_{k} \star \tilde G_{n-k} q^{2k-n}
= q
\sum_{k=0}^{n-1} W_{n-k} \star W_{-k} q^{n-1-2k}
= \lbrack 2 \rbrack^{n-1}_q (q^{-1}xy+qyx)(xy+yx)^{n-1},
\\
&
\sum_{k=0}^n  \tilde G_{k} \star  G_{n-k} q^{n-2k}
= q
\sum_{k=0}^{n-1} W_{n-k} \star W_{-k} q^{2k+1-n}
= \lbrack 2 \rbrack^{n-1}_q (xy+yx)^{n-1} (q^{-1}xy+qyx),
\\
&
\sum_{k=0}^n \tilde G_{k} \star G_{n-k} q^{2k-n}
= q
\sum_{k=0}^{n-1} W_{-k} \star W_{n-k} q^{2k+1-n}
= \lbrack 2 \rbrack^{n-1}_q  (qxy+q^{-1}yx)(xy+yx)^{n-1}. 
\end{align*}
\end{proposition}

\begin{proposition} 
\label{prop:WWEx}
For $n \in \mathbb N$,
\begin{align*}
&
\sum_{k=0}^n  W_{-k} \star W_{k-n} q^{2k-n} = 
\sum_{k=0}^n W_{-k} \star  W_{k-n} q^{n-2k} =
q \lbrack 2 \rbrack^{n+1}_q x(xy+yx)^n x,
\\
&
\sum_{k=0}^n  W_{k+1} \star W_{n-k+1} q^{2k-n} = 
\sum_{k=0}^n W_{k+1} \star  W_{n-k+1} q^{n-2k} =
q \lbrack 2 \rbrack^{n+1}_q y(xy+yx)^n y.
\end{align*}
\end{proposition}

\begin{proposition} 
\label{prop:GGEx}
For $n \geq 1$,
\begin{align*}
&
\sum_{k=0}^n  G_{k} \star G_{n-k} q^{2k-n} = 
\sum_{k=0}^n G_{k} \star  G_{n-k} q^{n-2k} =
 \lbrack 2 \rbrack^{n}_q y(xy+yx)^{n-1} x,
\\
&
\sum_{k=0}^n  \tilde G_{k} \star \tilde G_{n-k} q^{2k-n} = 
\sum_{k=0}^n \tilde G_{k} \star  \tilde G_{n-k} q^{n-2k} =
 \lbrack 2 \rbrack^{n}_q x(xy+yx)^{n-1} y. 
\end{align*}
\end{proposition}

\begin{proposition} 
\label{prop:GWAlt}
For $n \geq 1$,
\begin{align*}
&
\sum_{k=0}^n  G_{n-k} \star W_{-k} q^{n-2k} = 
 q \lbrack 2 \rbrack^{n}_q (q^{-1}xy+qyx)(xy+yx)^{n-1} x,
\\
& \qquad \qquad  
\sum_{k=0}^n   W_{-k} \star  G_{n-k} q^{2k-n} = 
q^{-1} \lbrack 2 \rbrack^{n}_q (qxy+q^{-1}yx)(xy+yx)^{n-1} x,
\\
&\sum_{k=0}^n  G_{n-k} \star W_{k+1} q^{2k-n} = 
 q^{-1} \lbrack 2 \rbrack^{n}_q y(xy+yx)^{n-1} (qxy+q^{-1}yx),
\\
& \qquad \qquad 
\sum_{k=0}^n   W_{k+1} \star G_{n-k} q^{n-2k} = 
q \lbrack 2 \rbrack^{n}_q y(xy+yx)^{n-1} (q^{-1}xy+qyx),
\\
&\sum_{k=0}^n \tilde G_{n-k} \star W_{-k} q^{2k-n} = 
 q^{-1} \lbrack 2 \rbrack^{n}_q x(xy+yx)^{n-1} (q^{-1}xy+qyx),
\\
& \qquad \qquad  
\sum_{k=0}^n   W_{-k} \star \tilde G_{n-k} q^{n-2k} = 
q \lbrack 2 \rbrack^{n}_q x(xy+yx)^{n-1} (qxy+q^{-1}yx),
\\
&
\sum_{k=0}^n \tilde G_{n-k} \star W_{k+1} q^{n-2k} = 
 q \lbrack 2 \rbrack^{n}_q (qxy+q^{-1}yx)(xy+yx)^{n-1} y,
\\
& \qquad \qquad  
\sum_{k=0}^n   W_{k+1} \star \tilde G_{n-k} q^{2k-n} = 
q^{-1} \lbrack 2 \rbrack^{n}_q (q^{-1}xy+qyx)(xy+yx)^{n-1}y.
\end{align*}
\end{proposition}

\section{Directions for future research}

\begin{problem}
\label{prob:123}
\rm Determine if the
relations in 
Proposition \ref{prop:rel3} can be obtained 
directly
from the relations in
Propositions
\ref{prop:rel1},
\ref{prop:rel2}.
\end{problem}

\begin{problem}
\label{prob:ok}
\rm
Consider the algebra $\mathcal U^+_q$ defined by 
generators
(\ref{eq:WWGG}) subject to the relations in
Propositions \ref{prop:rel1},
\ref{prop:rel2},
\ref{prop:rel3}. By those propositions
the algebra $\mathcal U^+_q$ is a homomorphic preimage of
$U^+_q$. Determine the center of $\mathcal U^+_q$
and the kernel of the homomorphism.
\end{problem}

\begin{problem}
\label{prob:WWGG}
\rm  Find some generators
$\lbrace  W_{-k}\rbrace_{k=0}^\infty$,
$\lbrace  W_{k+1}\rbrace_{k=0}^\infty$,
$\lbrace  G_{k+1}\rbrace_{k=0}^\infty$,
$\lbrace  {\tilde G}_{k+1}\rbrace_{k=0}^\infty$
for $\mathcal O_q$ that satisfy the relations in
\cite[Definition~3.1]{basnc}.
The papers
\cite{basBel},
\cite{pospart},
\cite{lusztigaut},
\cite{pbw},
\cite{boxq},
\cite{z2z2z2}
might be helpful in this direction.
\end{problem}

\begin{problem} \rm
In \cite{BK}
Baseilhac and Kolb
obtained a PBW basis for $\mathcal O_q$ that involves some
elements
\begin{align}
\label{eq:kolb}
\lbrace B_{n\delta+\alpha_0}\rbrace_{n=0}^\infty, \quad
\lbrace B_{n\delta+\alpha_1}\rbrace_{n=0}^\infty, \quad
\lbrace B_{n\delta}\rbrace_{n=1}^\infty.
\end{align}
This PBW basis is roughly
analogous to the PBW basis for $U^+_q$ given in
   (\ref{eqUq:PBWintro}).
Find the relationship between the elements 
(\ref{eq:kolb})
and
the $\mathcal O_q$ generators in Problem
\ref{prob:WWGG}.
We expect that the relationship generalizes the results
in Section 11 of the present paper.
\end{problem}

\section{Acknowledgment} 
The author thanks Pascal Baseilhac, Samuel Belliard,
Jonas Hartwig,
Jae-ho Lee, 
Kazumasa Nomura,
and Travis Scrimshaw for valuable discussions about this paper and
related topics.

\section{Appendix A: Commutator relations for alternating words, part I}

In this appendix we give a reformulation of Lemma
\ref{lem:transGW}.

\begin{lemma}
\label{lem:rr1S}
For $i,j\in \mathbb N$ the following holds in the $q$-shuffle 
algebra $\mathbb V$.
\begin{enumerate}
\item[\rm (i)]  For $i\leq j$,
\begin{align*}
 G_i \star W_{-j} &=
  W_{-j}\star G_i + 
(1-q^{2})\sum_{\ell=1}^i   W_{-\ell-j}\star G_{i-\ell}  
+ (q^{2}-1)\sum_{\ell=1}^i  W_{\ell-i}\star  G_{j+\ell},
\\
W_{-j}\star G_i &=
 G_i \star W_{-j}  + 
(1-q^{-2})\sum_{\ell=1}^i G_{i-\ell}\star W_{-\ell-j}
+ (q^{-2}-1)\sum_{\ell=1}^i G_{j+\ell} \star W_{\ell-i}
\end{align*}
\noindent and
\begin{align*}
G_i\star W_{j+1} &=
 W_{j+1} \star  G_i + 
(1-q^{-2})\sum_{\ell=1}^i W_{\ell+j+1}\star  G_{i-\ell}
+ (q^{-2}-1)\sum_{\ell=1}^i W_{i-\ell+1} \star  G_{j+\ell},
\\
 W_{j+1} \star  G_i &=
 G_i\star W_{j+1} + 
(1-q^{2})\sum_{\ell=1}^i  G_{i-\ell} \star W_{\ell+j+1}
+ (q^{2}-1)\sum_{\ell=1}^i  G_{j+\ell} \star W_{i-\ell+1}
\end{align*}
\noindent and
\begin{align*}
\tilde G_i\star W_{-j} &=
 W_{-j} \star \tilde G_i + 
(1-q^{-2})\sum_{\ell=1}^i W_{-\ell-j}\star \tilde G_{i-\ell}
+ (q^{-2}-1)\sum_{\ell=1}^i W_{\ell-i} \star \tilde G_{j+\ell},
\\
 W_{-j} \star \tilde G_i &=
\tilde G_i\star W_{-j} + 
(1-q^{2})\sum_{\ell=1}^i  \tilde G_{i-\ell} \star W_{-\ell-j}
+ (q^{2}-1)\sum_{\ell=1}^i  \tilde G_{j+\ell} \star W_{\ell-i}
\end{align*}
\noindent and
\begin{align*}
\tilde G_i \star W_{j+1} &=
  W_{j+1}\star \tilde G_i + 
(1-q^{2})\sum_{\ell=1}^i  W_{\ell+j+1} \star \tilde G_{i-\ell}  
+ (q^{2}-1)\sum_{\ell=1}^i W_{i-\ell+1}\star  \tilde G_{j+\ell},
\\
W_{j+1}\star \tilde G_i &=
\tilde G_i \star W_{j+1}  + 
(1-q^{-2})\sum_{\ell=1}^i \tilde G_{i-\ell}\star W_{\ell+j+1}
+ (q^{-2}-1)\sum_{\ell=1}^i \tilde G_{j+\ell} \star W_{i-\ell+1}.
\end{align*}
\item[\rm (ii)]  For $i>j$,
\begin{align*}
G_i \star W_{-j}  &=
q^{2} W_{-j}\star  G_i + 
(1-q^{2}) \sum_{\ell=0}^j W_{-i-\ell}\star  G_{j-\ell}
+
(q^{2}-1)\sum_{\ell=1}^j W_{\ell-j}\star G_{i+\ell},
\\
W_{-j}\star  G_i  &=
q^{-2} G_i \star W_{-j}  + 
(1-q^{-2}) \sum_{\ell=0}^j G_{j-\ell}\star W_{-i-\ell}
+
(q^{-2}-1)\sum_{\ell=1}^j G_{i+\ell}\star W_{\ell-j}
\end{align*}
\noindent and
\begin{align*}
 G_i \star W_{j+1} &=
q^{-2} W_{j+1} \star  G_i + 
(1-q^{-2}) \sum_{\ell=0}^j W_{i+\ell+1}\star  G_{j-\ell}
+
(q^{-2}-1)\sum_{\ell=1}^j W_{j-\ell+1}\star G_{i+\ell},
\\
W_{j+1} \star  G_i  &=
q^{2}  G_i \star W_{j+1}+ 
(1-q^{2}) \sum_{\ell=0}^j  G_{j-\ell} \star W_{i+\ell+1}
+
(q^{2}-1)\sum_{\ell=1}^j G_{i+\ell} \star W_{j-\ell+1}
\end{align*}
\noindent and
\begin{align*}
\tilde G_i \star W_{-j} &=
q^{-2} W_{-j} \star \tilde G_i + 
(1-q^{-2}) \sum_{\ell=0}^j W_{-i-\ell}\star \tilde G_{j-\ell}
+
(q^{-2}-1)\sum_{\ell=1}^j W_{\ell-j}\star \tilde G_{i+\ell},
\\
W_{-j} \star \tilde G_i  &=
q^{2} \tilde G_i \star W_{-j}+ 
(1-q^{2}) \sum_{\ell=0}^j \tilde G_{j-\ell} \star W_{-i-\ell}
+
(q^{2}-1)\sum_{\ell=1}^j \tilde G_{i+\ell} \star W_{\ell-j}
\end{align*}
\noindent and
\begin{align*}
\tilde G_i \star W_{j+1}  &=
q^{2} W_{j+1}\star \tilde G_i + 
(1-q^{2}) \sum_{\ell=0}^j W_{i+\ell+1}\star  \tilde G_{j-\ell}
+ (q^{2}-1)\sum_{\ell=1}^j W_{j-\ell+1}\star \tilde G_{i+\ell},
\\
W_{j+1}\star \tilde G_i  &=
q^{-2} \tilde G_i \star W_{j+1}  + 
(1-q^{-2}) \sum_{\ell=0}^j  \tilde G_{j-\ell}\star W_{i+\ell+1}
+
(q^{-2}-1)\sum_{\ell=1}^j\tilde G_{i+\ell}\star W_{j-\ell+1}.
\end{align*}
\end{enumerate}
\end{lemma}

\newpage
\section{Appendix B: 
Commutator relations for alternating words, part II}

In this appendix we give a reformulation of
Lemma \ref{lem:ggww}.

\begin{lemma}
\label{lem:rr2}
For $i,j\in \mathbb N$ the following holds
in the $q$-shuffle algebra $\mathbb V$.
\begin{enumerate}
\item[\rm (i)]  For $i\leq j$,
\begin{align*}
\tilde G_i \star G_j &= G_j \star \tilde G_i + 
(1-q^2)\sum_{\ell=1}^i W_{\ell-i} \star W_{j+\ell} 
-(1-q^2)\sum_{\ell=0}^{i-1} W_{-j-\ell} \star W_{i-\ell},
\\
G_i \star \tilde G_j &= \tilde G_j  \star G_i + 
(1-q^2)\sum_{\ell=1}^i W_{i-\ell+1} \star  W_{-j-\ell+1}
-(1-q^2)\sum_{\ell=0}^{i-1} W_{j+\ell+1}\star W_{\ell-i+1},
\\
 \tilde G_j \star  G_i &= G_i \star \tilde G_j  + 
(1-q^2)\sum_{\ell=1}^i W_{j+\ell} \star W_{\ell-i} 
-(1-q^2)\sum_{\ell=0}^{i-1} W_{i-\ell}\star W_{-j-\ell},
\\
G_j \star \tilde G_i   &=   \tilde G_i \star  G_j + 
(1-q^2)\sum_{\ell=1}^i W_{-j-\ell+1} \star W_{i-\ell+1}
-(1-q^2)\sum_{\ell=0}^{i-1} W_{\ell-i+1}\star W_{j+\ell+1},
\\
\\
W_{i+1} \star  W_{-j} &= W_{-j} \star W_{i+1} +
(1-q^{-2})\sum_{\ell=0}^i G_{j+1+\ell} \star   \tilde G_{i-\ell}
-(1-q^{-2})\sum_{\ell=0}^i G_{i-\ell} \star  \tilde G_{j+1+\ell},
\\
W_{-i}\star W_{j+1} &= W_{j+1} \star W_{-i} +
(1-q^{-2})\sum_{\ell=0}^i \tilde G_{j+1+\ell} \star  G_{i-\ell}
-(1-q^{-2})\sum_{\ell=0}^i \tilde G_{i-\ell} \star  G_{j+1+\ell},
\\
W_{-j} \star W_{i+1}  &=  W_{i+1} \star W_{-j} +
(1-q^{-2})\sum_{\ell=0}^i G_{i-\ell} \star \tilde G_{j+1+\ell}
-(1-q^{-2})\sum_{\ell=0}^i G_{j+1+\ell} \star \tilde G_{i-\ell},
\\
 W_{j+1} \star  W_{-i} &= W_{-i} \star W_{j+1}  +
(1-q^{-2})\sum_{\ell=0}^i  \tilde  G_{i-\ell} \star G_{j+1+\ell}
-(1-q^{-2})\sum_{\ell=0}^i  \tilde  G_{j+1+\ell} \star G_{i-\ell}.
\end{align*}
\item[\rm (ii)] For $i> j$,
\begin{align*}
\tilde G_i \star G_j &= G_j \star \tilde G_i + 
(1-q^2)\sum_{\ell=1}^j W_{\ell-j}\star  W_{i+\ell} 
-(1-q^2)\sum_{\ell=0}^{j-1} W_{-i-\ell} \star  W_{j-\ell},
\\
 G_i \star \tilde G_j &= \tilde G_j \star  G_i + 
(1-q^2)\sum_{\ell=1}^j W_{j-\ell+1} \star W_{-i-\ell+1}
-(1-q^2)\sum_{\ell=0}^{j-1} W_{i+\ell+1}\star  W_{\ell-j+1},
\\
\tilde G_j \star G_i &= G_i \star \tilde G_j + 
(1-q^2)\sum_{\ell=1}^j W_{i+\ell}\star  W_{\ell-j} 
-(1-q^2)\sum_{\ell=0}^{j-1} W_{j-\ell}\star W_{-i-\ell},
\\
G_j \star \tilde  G_i  &=   \tilde G_i \star G_j + 
(1-q^2)\sum_{\ell=1}^j W_{-i-\ell+1} \star W_{j-\ell+1}
-(1-q^2)\sum_{\ell=0}^{j-1} W_{\ell-j+1} \star  W_{i+\ell+1},
\\
\\
W_{i+1}\star  W_{-j} &= W_{-j}\star  W_{i+1} +
(1-q^{-2})\sum_{\ell=0}^j G_{i+1+\ell} \star   \tilde G_{j-\ell}
-(1-q^{-2})\sum_{\ell=0}^j G_{j-\ell} \star  \tilde G_{i+1+\ell},
\\
W_{-i} \star W_{j+1} &= W_{j+1}\star   W_{-i}+
(1-q^{-2})\sum_{\ell=0}^j \tilde G_{i+1+\ell}  \star G_{j-\ell}
-(1-q^{-2})\sum_{\ell=0}^j \tilde G_{j-\ell} \star  G_{i+1+\ell},
\\
W_{-j}\star W_{i+1}  &=  W_{i+1}\star W_{-j}+
(1-q^{-2})\sum_{\ell=0}^j G_{j-\ell} \star \tilde G_{i+1+\ell} 
-(1-q^{-2})\sum_{\ell=0}^j G_{i+1+\ell} \star \tilde G_{j-\ell},
\\
 W_{j+1}\star W_{-i} &= W_{-i}\star W_{j+1}+
(1-q^{-2})\sum_{\ell=0}^j    \tilde  G_{j-\ell}\star G_{i+1+\ell} 
-(1-q^{-2})\sum_{\ell=0}^j  \tilde  G_{i+1+\ell} \star G_{j-\ell}.
\end{align*}
\end{enumerate}
\end{lemma}

\newpage
\section{Appendix C: Examples of commutator relations }

\noindent In this appendix we give some examples of commutator relations 
from Lemmas \ref{lem:transGW} and
\ref{lem:ggww}.

\bigskip

\noindent Length 2:
\begin{align*}
W_1 \star W_0 &= W_0\star W_1 + (1-q^{-2})(G_1 - \tilde G_1).
\end{align*}

\noindent Length 3:
\begin{align*}
G_1\star W_0 &= q^{2} W_0 \star G_1 + (1-q^{2})W_{-1},
\\
W_1 \star G_1 &= q^{2}  G_1\star W_1 + (1-q^{2})W_{2},
\\
\tilde G_1 \star W_0 &= q^{-2} W_0 \star \tilde G_1 + (1-q^{-2})W_{-1},
\\
W_1 \star \tilde G_1 &= q^{-2} \tilde G_1 \star W_1 + (1-q^{-2})W_{2}.
\end{align*}

\noindent Length 4:
\begin{align*}
W_2 \star  W_0 &= W_0 \star W_2 + (1-q^{-2})(G_2 - \tilde G_2),
\\
W_1 \star W_{-1} &= W_{-1} \star  W_1 + (1-q^{-2})(G_2 - \tilde G_2),
\\
\tilde G_1 \star G_{1} &= G_1 \star \tilde G_{1} + (1-q^{2})(W_0 \star W_2
- W_{-1}\star W_1).
\end{align*}

\noindent Length 5:
\begin{align*}
 G_2 \star W_0 &= 
q^{2} W_0 \star G_2 + (1-q^{2})W_{-2},
\\
 G_1 \star W_{-1} &= (q^{2}-1)W_0 \star  G_2 +W_{-1} \star G_1
+(1-q^{2})W_{-2},
\\
W_1\star  G_2 &= 
q^{2}  G_2 \star  W_1+ (1-q^{2})W_{3},
\\
W_2 \star  G_1&= 
(q^{2}-1) G_2\star W_1 + G_1 \star  W_2
+(1-q^{2})W_3,
\\
\tilde G_2\star W_0 &= 
q^{-2} W_0 \star \tilde G_2 + (1-q^{-2})W_{-2},
\\
\tilde G_1 \star W_{-1} &= (q^{-2}-1)W_0 \star \tilde G_2 +W_{-1}\star \tilde G_1
+(1-q^{-2})W_{-2},
\\
W_1 \star  \tilde G_2 &= 
q^{-2} \tilde G_2 \star W_1+ (1-q^{-2})W_{3},
\\
W_2 \star \tilde G_1&= 
(q^{-2}-1) \tilde G_2 \star W_1 +\tilde G_1\star W_2
+(1-q^{-2})W_3.
\end{align*}

\noindent Length 6:
\begin{align*}
W_3 \star W_0 &= W_0\star W_3 + (1-q^{-2})(G_3 - \tilde G_3),
\\
W_2 \star W_{-1} &= W_{-1} \star W_2 +
(1-q^{-2})(G_2\star \tilde G_1 + G_3  -G_1 \star \tilde G_2
- \tilde G_3),
\\
W_1 \star W_{-2} &= W_{-2} \star W_1 + (1-q^{-2})(G_3 - \tilde G_3),
\\
\tilde G_1 \star G_{2} &= G_2 \star \tilde G_{1}
+ (1-q^{2})(W_0 \star W_3 - W_{-2} \star W_1),
\\
\tilde G_2 \star G_{1} &= G_1 \star \tilde G_{2} + 
(1-q^{2})(W_0 \star W_3 - W_{-2}\star W_1).
\end{align*}

\noindent Length 7:
\begin{align*}
& G_3 \star W_0= 
q^{2}W_0 \star  G_3+(1-q^{2})W_{-3},
\\
& G_2 \star W_{-1}= 
q^{2} W_{-1} \star G_2+
(1-q^{2})W_{-2}\star G_1+
(q^{2}-1)W_{0}\star G_3+
(1-q^{2})W_{-3},
\\
& G_1\star W_{-2}= 
 W_{-2}\star G_1+
(q^{2}-1)W_{0}\star G_3+
(1-q^{2})W_{-3},
\\
&W_1 \star G_3= 
q^{2} G_3\star W_1+(1-q^{2})W_{4},
\\
&W_2 \star G_2= 
q^{2}  G_2\star W_2+
(1-q^{2}) G_1 \star W_3+
(q^{2}-1) G_3 \star W_1+
(1-q^{2})W_{4},
\\
&W_3 \star G_1= 
  G_1 \star W_3+
(q^{2}-1) G_3\star W_1+
(1-q^{2})W_{4},
\\
&\tilde G_3 \star W_0= 
q^{-2}W_0 \star \tilde G_3+(1-q^{-2})W_{-3},
\\
&\tilde G_2 \star W_{-1}= 
q^{-2} W_{-1}\star \tilde G_2+
(1-q^{-2})W_{-2} \star \tilde G_1+
(q^{-2}-1)W_{0}\star \tilde G_3+
(1-q^{-2})W_{-3},
\\
&\tilde G_1\star  W_{-2}= 
 W_{-2}\star \tilde G_1+
(q^{-2}-1)W_{0}\star \tilde G_3+
(1-q^{-2})W_{-3},
\\
&W_1 \star \tilde G_3= 
q^{-2}\tilde G_3\star W_1+(1-q^{-2})W_{4},
\\
&W_2 \star \tilde G_2= 
q^{-2} \tilde G_2\star W_2+
(1-q^{-2})\tilde G_1\star  W_3+
(q^{-2}-1)\tilde G_3 \star W_1+
(1-q^{-2})W_{4},
\\
&W_3\star \tilde G_1= 
 \tilde G_1 \star W_3+
(q^{-2}-1)\tilde G_3 \star W_1+
(1-q^{-2})W_{4}.
\end{align*}

\noindent Length 8:
\begin{align*}
W_4 \star W_0 &=W_0 \star W_4 +(1-q^{-2})(G_4 -\tilde G_4),
\\
W_3 \star W_{-1} &=W_{-1}\star W_3 +(1-q^{-2})(G_3 \star 
\tilde G_1 + G_4 - G_1 \star \tilde G_3 -\tilde G_4),
\\
W_2 \star W_{-2} &=W_{-2} \star W_2 +(1-q^{-2})(G_3\star \tilde G_1
+ G_4 - G_1 \star \tilde G_3 -\tilde G_4),
\\
W_1 \star W_{-3} &=W_{-3} \star W_1 +(1-q^{-2})(G_4 -\tilde G_4),
\\
\tilde G_1 \star G_3 &= G_3 \star \tilde G_1 +
(1-q^2)(W_0\star W_4-W_{-3}\star W_1),
\\
\tilde G_2 \star G_2 &= 
G_2 \star \tilde G_2 + (1-q^2)(W_0\star W_4+W_{-1}\star W_3-
W_{-2}\star W_2-W_{-3}\star W_1),
\\
\tilde G_3 \star G_1 &= G_1 \star \tilde G_3 + 
(1-q^2)(W_0\star W_4-W_{-3}\star W_1).
\end{align*}

\bigskip

\noindent Paul Terwilliger \hfil\break
\noindent Department of Mathematics \hfil\break
\noindent University of Wisconsin \hfil\break
\noindent 480 Lincoln Drive \hfil\break
\noindent Madison, WI 53706-1388 USA \hfil\break
\noindent email: {\tt terwilli@math.wisc.edu }\hfil\break

\end{document}